\documentclass[12pt, reqno]{amsart}
\usepackage{amsfonts}
\usepackage{amsmath}
\usepackage{mathrsfs}
\usepackage{enumitem}
\usepackage{amsmath, amsthm, amscd, amsfonts, amssymb, graphicx, color}
\usepackage[bookmarksnumbered, colorlinks, plainpages]{hyperref}
\hypersetup{colorlinks=true,linkcolor=blue, anchorcolor=blue,
    citecolor=cyan, urlcolor=blue, filecolor=magenta, pdftoolbar=true}
\usepackage{color}

\def\Xint#1{\mathchoice
    {\XXint\displaystyle\textstyle{#1}}%
    {\XXint\textstyle\scriptstyle{#1}}%
    {\XXint\scriptstyle\scriptscriptstyle{#1}}%
    {\XXint\scriptscriptstyle\scriptscriptstyle{#1}}%
    \!\int}
\def\XXint#1#2#3{{\setbox0=\hbox{$#1{#2#3}{\int}$ }
        \vcenter{\hbox{$#2#3$ }}\kern-.6\wd0}}

\def\dashint{\Xint-}

\newtheorem{theorem}{Theorem}[section]
\newtheorem{lemma}[theorem]{Lemma}

\newtheorem{corollary}[theorem]{Corollary}
\theoremstyle{definition}
\newtheorem{definition}[theorem]{Definition}

\theoremstyle{remark}
\newtheorem{remark}[theorem]{Remark}
\numberwithin{equation}{section}

\makeatletter

\newcommand{\Rmnum}[1]{\expandafter\@slowromancap\romannumeral #1@}
\makeatother

\begin{document}
\allowdisplaybreaks[1]

\title[ Fractional Equations in Grushin-type Spaces ]{Local Behavior of Fractional  Integro-Differential Equations in Grushin-type Spaces }
\author[B. Xu]{Boxiang Xu}
\address{School of Mathematics and Physics, University of Science and Technology Beijing, Beijing 100083, China}
\email{cellinia1901@163.com}

\author[Y. Liu]{Yu Liu}
\address{School of Mathematics and Physics, University of Science and Technology Beijing, Beijing 100083, China}
\email{liuyu75@pku.org.cn} \thanks {The corresponding author is Yu Liu, E-mail: liuyu75@pku.org.cn}

\author[S. Shi]{Shaoguang Shi}
\address{School of Mathematics and Statistics, Linyi University, Linyi 276005, China}
\email{shishaoguang@mail.bnu.edu.cn}

\maketitle

{\Small {\bf Abstract:} \ In this paper, we develop the De
Giorgi-Nash-Moser  theory for a broad class of nonlinear equations
governed by nonlocal, possibly degenerate, integro-differential
operators. The fractional $p$-Laplacian operator, in the context of
Grushin-type spaces $\mathbb{G}^n$, serves as a prototypical
example. Among our findings, we demonstrate that the weak solutions
to these problems are both bounded and H\"{o}lder continuous.
Additionally, we establish general estimates, including fractional
Caccioppoli-type estimates with tail terms and logarithmic-type
bounds.

    \vspace{0.2cm} {\bf Keywords:} {quasilinear nonlocal operators, H{\"o}lder continuity, Grushin-type space, fractional Grushin operator}

    }
\tableofcontents

\begin{section}{Introduction}\label{section_Introduction}
This paper aims to establish some regularity results for nonlocal integro-differential operators and minimizers of fractional order $s \in (0,1)$  and summability $p>1$ in Grushin-type spaces, which is denoted as ${\mathbb{G}}^{n}$ in our paper. Let $\Omega$ be an open, bounded subset of the Grushin-type spaces ${\mathbb{G}}^{n}$ and consider $g$ as a function belonging to the fractional Sobolev space $W^{s,p}\left( {{\mathbb{G}}^{n}} \right)$. We will establish general local regularity estimates for the minimizers $u$, for which $u$ minimizes the functional
\begin{equation}\label{F function}
    \mathcal{F}(v):=\int_{{\mathbb{G}}^{n}} \int_{{\mathbb{G}}^{n}} {{\left| {d_{C}(x,y)} \right|}^{-Q-sp}{\left| {v(x)-v(y)} \right|}^{p}} \mathrm{d}x\mathrm{d}y,
\end{equation}
over the class of functions $\left\{ {v \in W^{s,p}\left({{\mathbb{G}}^{n}} \right): v=g \;\:\mathrm{a.e.\in}\;{\mathbb{G}}^{n}\backslash \Omega} \right\}$, where $Q=n+m(n-h)$ is the usual homogeneous dimension of the Grushin-type spaces ${\mathbb{G}}^{n}$ which will be specifically introduced in Section \ref{section_Preliminaries}. Here, $d_{C}(\cdot ,\cdot)$ represents the Carnot-Carath\'eodory distance in ${\mathbb{G}}^{n}$, which will be specifically introduced in Section \ref{section_Preliminaries} (see  \eqref{Carnot-Caratheodory distance}). As demonstrated in Theorem  \ref{theo_Equivalence between weak solutions and the minimizers} below, minimizers are equivalently characterized as the weak solutions to the following integro-differential problems
\begin{equation}\label{integro-differential problem}
    \left \{
    \begin{aligned}
        &\mathcal{L}u=0&&\mathrm{in}\,\,\Omega,   \\
        &u=g &&\mathrm{in}\,\,{\mathbb{G}}^{n}\backslash\Omega,
    \end{aligned}
    \right.
\end{equation}
in which the operator $\mathcal{L}$ is formally defined as
\begin{equation}\label{nonlocal operator L}
    \mathcal{L}u(x)=P.V.\int_{{\mathbb{G}}^{n}} {\frac{{\left| {u(x)-u(y)} \right|}^{p-2}(u(x)-u(y))}{{\left| {d_{C}(x,y)} \right|}^{Q+sp}}}\mathrm{d}y, \quad\, x\in{\mathbb{G}}^{n},
\end{equation}
where the symbol $P.V.$ means in the principal value sense.

Recently, the investigation of problems involving fractional Sobolev spaces  and the corresponding nonlocal equations has garnered significant attention, both from a purely mathematical standpoint and in specific applications. This interest stems from the natural occurrence of such problems in various contexts. Despite the relatively short history of this field, the existing literature is extensive, precluding a comprehensive treatment in a single paper. For an introductory overview of this topic and a list of relevant references, see  Nezza-Palatucci-Valdinoci's work in \cite{di2012hitchhiker} and some supplementary results when studying other issues, see e.g. \cite{SX1,SX2,Wa}.

Regarding regularity and related results, the theory in the fractional Euclidean case has recently made significant strides. Even in the nonlinear setting, where $p\ne 2$, substantial progress has been made for the Euclidean analogues of \eqref{integro-differential problem}, where ${\left| {x-y} \right|}^{-n-sp}$ or more general expressions replace ${\left| {d_{C}(x,y)} \right|}^{-Q-sp}$. This includes findings on the boundedness, Harnack inequalities, and H{\"o}lder continuity (up to the boundary) for the fractional $p$-Dirichlet problems, as documented in \cite{brasco2018higher,di2014nonlocal,di2016local,iannizzotto2016global,korvenpaa2016obstacle,SZW}. Moreover, the equivalence of solutions defined through integration by parts with test functions, as viscosity solutions, and via comparison to \eqref{integro-differential problem} has been established by Korvenp \cite{korvenpaa2019equivalence} and Lindgren and Lindqvist \cite{lindgren2016perron}, some relevent findings, see also \cite{SZZ,S}. For a comprehensive overview of these and related findings, please consult the work of Palatucci \cite{palatucci2018dirichlet}.

In recent years, analogous fractional frameworks have yielded a wealth of corresponding results in the context of the Heisenberg group. In \cite{roncal2016hardy}, Roncal and Thangavelu offer an explicit integral definition for the case $p=2$
and establish several Hardy inequalities for the conformally invariant fractional powers of the sub-Laplacian, thereby extending some of the important findings from \cite{frank2008hardy} by Frank, Lieb, and Seiringer to the Heisenberg group. Furthermore, the Heisenberg type uncertainty inequalities proved for the sub-Laplacian in Danielli's \cite{danielli2011hardy} are extended to the fractional framework. Similar conclusions to those observed in the Euclidean case have also been established in the Heisenberg group, and a few of these are briefly outlined here. In \cite{manfredini2023holder}, Manfredini establishes the boundedness and H{\"o}lder continuity of weak solutions to similar integro-differential problems as \eqref{integro-differential problem} by deriving fractional Caccioppoli-type estimates with tail and logarithmic-type estimates. These results were further generalized to more general cases in \cite{fang2024regularity} by Fang and Zhang. Palatucci and Piccinini prove general Harnack inequalities for the related weak solutions in \cite{palatucci2022nonlocal}. In \cite{piccinini2022obstacle}, Piccinini demonstrates the existence and uniqueness of the solution to the obstacle problem associated with a broad class of nonlinear integro-differential operators, and that solutions inherit the regularity properties of the obstacle, including boundedness, continuity, and H{\"o}lder continuity up to the boundary.

In the context of the fractional framework, the derivation of the desired local estimates necessitates the exploration of nonlocal estimates. To this end, we would like to emphasize a particular quantity that will serve as the foundation of this paper. Namely, we introduce the nonlocal tail of a function $v\in W^{s,p}\left( {{\mathbb{G}}^{n}} \right)$ in the ball $B\left( {x_{0},R} \right) \subset {\mathbb{G}}^{n}$ as
\begin{equation}\label{nonlocal tail}
    \mathrm{Tail} \left( {v;x_{0},R} \right) :=\left[ {R^{sp}\int_{{\mathbb{G}}^{n}\backslash B\left( {x_{0},R} \right)} {{\left| {v(x)} \right|}^{p-1}\left| {d_{C}\left( {x,x_{0}} \right)} \right|^{-Q-sp}}\mathrm{d}x} \right]^{\frac{1}{p-1}},
\end{equation}
where $Q$ is the
homogeneous dimension of ${\mathbb{G}}^{n}$ (see Subsection \ref{subsection_Grushin-type Spaces and Notation}). In the standard Euclidean framework, the nonlocal tail has already been established as a critical component in proofs requiring precise quantitative control of the long-range interactions that naturally emerge when dealing with nonlocal operators, as seen in \eqref{nonlocal operator L}; see e.g. the work of Castro, Kuusi, and Palatucci in \cite{di2016local}.

At this stage, we  present  our first result  as follows.

\begin{theorem}[Local boundedness]\label{theo_local boundedness}
    Let $p\in (\max\left\{ {m(n-h),1} \right\},+\infty)$, $s\in \left( {0,1-\frac{m(n-h)}{p}} \right)$, $B\left( {x_{0},r} \right) \subset \Omega$,
    where $m,n,h$ appear in Subsection \ref{subsection_Grushin-type Spaces and Notation},  and let $u\in W^{s,p}\left( {{\mathbb{G}}^{n}} \right)$
    be a weak subsolution to problem \eqref{integro-differential problem}.  Then the following estimate
    \begin{equation}\label{local boundedness}
        \mathop{\sup}_{B\left( {x_{0},r/2} \right)} u\le \delta \frac{\mathrm{Tail} \left( {u_{+};x_{0},r/2} \right)} {r^{\frac{n-Q}{p-1}}\left( {{\left| {x_{0}} \right|_h}+r} \right)^{\frac{m(n-h)}{p-1}}}+c\delta ^{-\frac{(2n-Q)(p-1)}{sp^2}} {\left( {\dashint_{B\left( {x_{0},r} \right)} {u_{+}^{p}}\mathrm{d}x} \right)}^{\frac{1}{p}}
    \end{equation}
    holds, where $u_{+}:=\mathrm{max}\left\{ {u,0} \right\}$, $\delta \in\left( {0,1} \right]$, $\left| {x_{0}} \right|_h$ is the distance  in \eqref{Grushin-type spaces} in Section \ref{section_Preliminaries}, and the constant $c$ depends only on $n,h,p,s,m$ and $\Omega$.
\end{theorem}

It should be noted that the above theorem reduces to the case of \cite[Theorem 1.1]{di2016local} when the parameter $m$ is zero.

Utilizing the Logarithmic Lemma and the Caccioppoli estimate with tail, along with the derived local boundedness, we establish our second result, which is the following H{\"o}lder continuity theorem.

\begin{theorem}[H{\"o}lder continuity]\label{theo_holder continuity}
    Let $p\in (\max\{m(n-h),1\},+\infty)$, $s\in \left( {0,1-\frac{m(n-h)}{p}} \right)$ and let $u\in W^{s,p}\left( {{\mathbb{G}}^{n}} \right)$ be a solution to the problem \eqref{integro-differential problem}. Then $u$ is locally  H{\"o}lder continuous in $\Omega$. In particular, there are positive constants $\alpha <sp/(p-1)$ and $c>0$, both depending only on $n,h,p,s,m,\Omega$ and $\|u\|_{L^\infty(B(x_0,r))}$, such that if $B\left( {x_{0},2r} \right) \subset \Omega$, then
    \begin{equation*}
        \mathop{\mathrm{osc}}_{B\left( {x_{0},\varrho } \right)} u \le c\left( {\frac{\varrho}{r}} \right)^{\alpha}\left[ {\frac{\mathrm{Tail} \left( {u;x_{0},r} \right)}{r^{\frac{n-Q}{p-1}}\left( {{\left| {x_{0}} \right|_h}+r} \right)^{\frac{m(n-h)}{p-1}}}+{\left( {\dashint_{B\left( {x_{0},2r} \right)} {\left| {u} \right|^{p}}\mathrm{d}x} \right)}^{\frac{1}{p}}} \right]
    \end{equation*}
    holds for any $\varrho \in \left( {0,r} \right]$.
\end{theorem}


For the proofs of the H{\"o}lder continuity and the local boundedness, the precise estimates play a critical role, specifically the Caccioppoli-type estimate (see Theorem \ref{theo_Caccioppoli estimates with tail} below) and the logarithmic-type estimate (see forthcoming Lemma \ref{lemma_Logarithmic Lemma}). Next,  we demonstrate that the fractional $p\text{-}$minimizers, which are  equivalent to  the weak solutions to the Euler-Lagrange equation corresponding to \eqref{F function}, satisfy the following nonlocal Caccioppoli-type inequalities.

\begin{theorem}[Caccioppoli estimates with tail]\label{theo_Caccioppoli estimates with tail}
    Let $p\in (1,+\infty)$, $s\in(0,1)$ and let $u\in W^{s,p}\left( {{\mathbb{G}}^{n}} \right)$ be a weak solution to \eqref{integro-differential problem}. Then  the following estimate
    \begin{equation}\label{Caccioppoli estimates with tail}
        \begin{aligned}
            &\int_{B_{r}}\int_{B_{r}} {\left| {d_{C}(x,y)} \right|^{-Q-sp}\left| {w_{\pm}(x)\phi(x)-w_{\pm}(y)\phi(y)} \right|^{p}}\mathrm{d}x\mathrm{d}y   \\
            \le &c\int_{B_{r}}\int_{B_{r}} {\left| {d_{C}(x,y)} \right|^{-Q-sp}{\left( {\mathrm{max}\left\{ {w_{\pm}(x),w_{\pm}(y)} \right\}} \right)}^{p}\left| {\phi(x)-\phi(y)} \right|^{p}}\mathrm{d}x\mathrm{d}y   \\
            &+c\int_{B_{r}} {w_{\pm}(x)\phi^{p}(x)}\mathrm{d}x\left( {\mathop{\sup}_{y\in \mathrm{{supp}\,\phi}}\int_{{\mathbb{G}}^{n} \backslash B_{r}} {\left| {d_{C}(x,y)} \right|^{-Q-sp}w_{\pm}^{p-1}(x)}\mathrm{d}x}
            \right)
        \end{aligned}
    \end{equation}holds for any $B_{r}\equiv B\left( {x_{0},r} \right) \subset \Omega$
    and any nonnegative $\phi \in C_{0}^{\infty}\left( {B_{r}} \right)$, where $w_{\pm}:=(u-k)_{\pm}$ and $c$ depends only on $p$.
\end{theorem}

In the proof of the H{\"o}lder continuity, the following logarithmic estimate plays a pivotal role.

\begin{theorem}[Logarithmic Lemma]\label{lemma_Logarithmic Lemma}
    Let $p\in (\max\left\{ {m(n-h),1} \right\},+\infty)$, $s\in \left( {0,1-\frac{m(n-h)}{p}} \right)$ and let $u\in W^{s,p}\left( {{\mathbb{G}}^{n}} \right)$ be a weak supersolution to \eqref{integro-differential problem} such that $u\ge 0$ in $B_{R}\equiv B\left( {x_{0},R} \right) \subset \Omega$. Then the following estimate holds for any $B_{r}\equiv B\left( {x_{0},r} \right) \subset B\left( {x_{0},R/2} \right)$ and any $d>0$,
    \begin{equation}\label{Logarithmic Lemma}
        \begin{aligned}
            &\int_{B_{r}}\int_{B_{r}} {\left| {d_{C}(x,y)} \right|^{-Q-sp}\left| {\mathrm{log}\left( {\frac{u(x)+d}{u(y)+d}} \right)}
                \right|^{p}}\mathrm{d}x\mathrm{d}y   \\
            \le &cr^{2n-Q-sp}\left( {\left| {x_0} \right|_h+2r} \right)^{m(n-h)}   \\
            &\times\left\{ {d^{1-p}{\left( {\frac{r}{R}} \right)}^{sp}r^{Q-n}\left[ {\mathrm{Tail}\left( {u_{-};x_{0},R}
                    \right)} \right]^{p-1}+\left( {\left| {x_0} \right|_h+2r} \right)^{m(n-h)}} \right\}
        \end{aligned}
    \end{equation}
    where $u_{-}:=\mathrm{max}\left\{ {-u,0} \right\}$ and $c$ depends only on $n,h,p,s$ and $m$.
\end{theorem}
\begin{remark}
    The estimate in \eqref{Caccioppoli estimates with tail} remains valid for $w_{+}$ when $u$ is a weak subsolution to \eqref{integro-differential problem}, and for $w_{-}$ when $u$ is a weak supersolution to \eqref{integro-differential problem}.
\end{remark}
\begin{remark}
    As shown in \eqref{Grushin-type spaces} in Section \ref{section_Preliminaries},  the Grushin-type space reduces
     to the general Euclidean space when $m=0$ or $n=h$. Similarly, the results of Theorem \ref{theo_local boundedness},
      Theorem \ref{theo_holder continuity} and Theorem \ref{lemma_Logarithmic Lemma} coincide with those presented in
      \cite{di2016local} when $m=0$ or $n=h$. Compared with the corresponding results in the Euclidean space in
      \cite{di2016local} and the Heisenberg groiup in \cite{manfredini2023holder}, an additional
      term $\left( {\left| {x_0} \right|_h+2r} \right)^{m(n-h)}$  arises  due to the fact that the measure
      of a ball in the Grusin-type space is related to its center coordinates, which is the most significant
       difference between our results and those of Euclidean space and Heisenberg group.
\end{remark}
The structure of the paper is as follows. In Section \ref{section_Preliminaries}, we introduce notation and provide some preliminary results. Section \ref{section_Fundamental estimates} focuses on proving the Caccioppoli estimates with tail in Theorem \ref{theo_Caccioppoli estimates with tail} and the Logarithmic Lemma in Theorem \ref{lemma_Logarithmic Lemma}. In Section \ref{section_Proof of the Local Boundedness}, we establish the local boundedness described in Theorem \ref{theo_local boundedness}. Finally, in Section \ref{section_Proof of the Hölder continuity}, we conclude by proving the H{\"o}lder continuity given by Theorem \ref{theo_holder continuity}.


\end{section}
\begin{section}{Preliminaries}\label{section_Preliminaries}
In this section, we outline the general assumptions of  the problem addressed in the present paper  and present notation, definitions, along with basic preliminary results that will be used throughout the subsequent sections.

\begin{subsection}{Grushin-type  spaces and notation}\label{subsection_Grushin-type Spaces and Notation}
We begin by providing some basic facts for the Grushin-type spaces. For more details on this space, see e.g. \cite{bieske2006P}.

Let ${\mathbb{R}}^{n}={\mathbb{R}}^{h}\times {\mathbb{R}}^{n-h}$, where $n\ge h\ge1$ are integers. Given a fixed nonnegative integer $m $ and $x=\left( {x_1,\dots,x_h,x_{h+1},\dots,x_n} \right)\in\mathbb{R}^n$, we consider the following vector fields
\begin{equation}\label{Grushin-type spaces}
    \left \{
    \begin{aligned}
        &X_{i}=\frac{\partial}{\partial x_{1}} &&\forall\,i\in\left\{ {1,2,\dots,h} \right\},   \\
        &X_{h+j}=\left| {x} \right|^m_h\frac{\partial}{\partial x_{i}} &&\forall\,j\in\left\{ {1,2,\dots,n-h} \right\},  \\
        &\left| x \right|_h=\left( {\sum_{i=1}^{h}x_i^2} \right)^{\frac{1}{2}},   \\
        &Q=n+m(n-h),\quad 2n>Q.
    \end{aligned}
    \right.
\end{equation}
The Euclidean space ${\mathbb{R}}^{n}$ endowed with the Carnot-Carath\'eodory distance $d_{C}$ associated to the above vector fields is called the Grushin-type space ${\mathbb{G}}^{n}$, where the Carnot-Carath\'eodory distance is defined for the points $p$ and $q$ as follows:
\begin{equation}\label{Carnot-Caratheodory distance}
    d_{C}( {p,q} )=\mathop{\inf}_{\Gamma}\int_{0}^{1}{\left\| {\gamma'( {t} )} \right\| dt}.
\end{equation}
Here $\Gamma$ is the set of all curves $\gamma$ such that $\gamma (0)=p$, $\gamma (1)=q$ and
\begin{equation*}
    \gamma'(t) \in \mathrm{span}\left\{ {\left\{ {X_{i}(\gamma(t))} \right\}^n_{i=1}} \right\}.
\end{equation*}
Via this metric, we can define a Carnot-Carath\'eodory ball of radius $r$ centered at a point $x_{0}$ by
\begin{equation*}
    B\left( {{x_{0},r}} \right)=\left\{ {{x\in \mathbb{G}^{n}:d_{C}( {x,x_{0})<r} )}} \right\}
\end{equation*}
and similarly, we shall denote an  open domain in $\mathbb{G}^{n}$ by $\Omega$. Given a smooth function $f$ on $\mathbb{G}^{n}$, we define the horizontal gradient of $f$ as
\begin{equation*}
    \nabla _{\mathbb{G}^{n}} f(x)=\left( {X_{1}f(x),X_{2}f(x),\cdots,X_{n}f(x)} \right).
\end{equation*}
Similarly to the Euclidean space, the Sobolev space $W^{1,p}\left( {\mathbb{G}^{n}} \right)$ for the Grushin-type $\mathbb{G}^{n}$ with $1\le p<\infty $ is defined as the set of all functions $f\in L^p$ such that $X_{i}f\in L^p$, $i=1,\cdots,n$. The norm of $f\in W^{1,p}\left( {\mathbb{G}^{n}} \right)$ is defined by
\begin{equation*}
    {\left\| {f} \right\|}_{W^{1,p}\left( {\mathbb{G}^{n}} \right)}=\sum_{i=1}^{n} {{\left\| {X_{i}f} \right\|}_{p}}+{\left\| {f} \right\|}_{p}.
\end{equation*}

The following lemma  concerning the measure of the ball $B\left( {{x_{0},r}} \right)$ in Grushin-type space plays  a pivotal role in the subsequent proofs. For a detailed proof, refer to \cite{liu2018functional}.
\begin{lemma}\label{lemma_measure of ball in Grushin spaces}
    For the ball $B\left( {{x_{0},r}} \right) \subseteq \mathbb{G}^{n}$, we have
    \begin{equation*}
        \left| {B\left( {{x_{0},r}} \right)} \right| \sim r^{n}\left( \left| {x_{0}} \right|_h+r \right)^{m(n-h)}.
    \end{equation*}
\end{lemma}

Before proceeding with the proofs, it is  necessary to establish some notation that will be used throughout the rest of the paper. First, we follow the standard convention of denoting by $c$ a general positive constant, which may vary across occurrences and may change from line to line. For clarity, the dependencies of these constants will often be omitted in chains of estimates and will be stated after the estimate. Throughout this paper, all notations will be listed as needed.
\begin{itemize}
    \item We denote by $B\left( {{x_{0},R}} \right)=\left\{ {{x\in \mathbb{G}^{n}:d_{C}( {x,x_{0})<R} )}} \right\}$ the open ball centered in $x_{0}\in \mathbb{G}^{n}$ with radius $R>0$. When the full expression is not essential and is evident from the context, we will adopt the shorter notation $B_{R}=B\left( {{x_{0},R}} \right)$. We denote by $\beta B_{R}$ the concentric ball scaled by a factor $\beta >0$, that is $\beta B_{R}:=B\left( {x_0,\beta R} \right)$.
    \item If $f\in L^{1}(S)$ and the measure $\left| S \right|$ of the set $S\subseteq \mathbb{G}^{n}$ is finite and strictly positive, we write
    \begin{equation*}
        (f)_{S}:=\dashint_{S} {f(x)}\mathrm{d}x=\frac{1}{\left| S \right|}\int_{S} {f(x)}\mathrm{d}x.
    \end{equation*}
    \item Let $k\in \mathbb{R}^{n}$,   denote by
    \begin{equation*}
        w_{+}(x):=\left( {u(x)-k} \right)_{+}=\max\left\{ {u(x)-k,0} \right\},
    \end{equation*}
    and
    \begin{equation*}
        w_{-}(x):=\left( {u(x)-k} \right)_{-}=\left( {k-u(x)} \right)_{+}.
    \end{equation*}
    \item The symbol $\thicksim$ between two positive expressions $u,v$ means that their ratio $\frac{u}{v}$ is bounded from above and below by positive constants.
\end{itemize}
\end{subsection}

\begin{subsection}{Fractional Sobolev space on the Grushin-type space}\label{subsection_The Setting of the Main Problem}
First, we need to give some definitions and basic results regarding our fractional functional setting. Let $p\ge 1$ and $s\in (0,1)$, and let $u:\mathbb{G}^{n}\to \mathbb{R}$ be a measurable function, the Gagliardo seminorm of $u$ can be defined as follows,
\begin{equation*}
    [u]_{W^{s,p}\left( {\mathbb{G}^n} \right)}=\left( {\int_{\mathbb{G}^{n}}\int_{\mathbb{G}^{n}} {\frac{\left| {u(x)-u(y)} \right|^p}{\left| {d_{C}(x,y)} \right|^{Q+sp}}}\mathrm{d}x\mathrm{d}y} \right)^{\frac{1}{p}}.
\end{equation*}
The fractional Sobolev spaces $W^{s,p}\left( {\mathbb{G}^n} \right)$ on the Grushin-type space are defined as
\begin{equation*}
    W^{s,p}\left( {\mathbb{G}^n} \right):=\left\{ {u\in L^{p}\left( {\mathbb{G}^n} \right):[u]_{W^{s,p}\left( {\mathbb{G}^n} \right)}<\infty} \right\},
\end{equation*}
which were endowed with the natural fractional norm
\begin{equation*}
    \left\| u \right\|_{W^{s,p}\left( {\mathbb{G}^n} \right)}:=\left( {\left\| u \right\|^p_{ L^{p}\left( {\mathbb{G}^n} \right)} +[u]^p_{W^{s,p}\left( {\mathbb{G}^n} \right)}} \right)^{\frac{1}{p}} \quad \forall u\in W^{s,p}\left( {\mathbb{G}^n} \right).
\end{equation*}
Similarly, given a domain $\Omega \subset \mathbb{G}^n$, one can define the fractional Sobolev space $W^{s,p}\left( {\Omega} \right)$ in the natural way, as follows
\begin{equation*}
    W^{s,p}\left( {\Omega} \right):=\left\{ {u\in L^{p}\left( {\Omega} \right):\left( {\int_{\Omega}\int_{\Omega} {\frac{\left| {u(x)-u(y)} \right|^p}{\left| {d_{C}(x,y)} \right|^{Q+sp}}}\mathrm{d}x\mathrm{d}y} \right)^{\frac{1}{p}}<\infty} \right\}
\end{equation*}
endowed with the norm
\begin{equation*}
    \left\| u \right\|_{W^{s,p}\left( {\Omega} \right)}:=\left( {\left\| u \right\|^p_{ L^{p}\left( {\Omega} \right)}+\int_{\Omega}\int_{\Omega} {\frac{\left| {u(x)-u(y)} \right|^p}{\left| {d_{C}(x,y)} \right|^{Q+sp}}}\mathrm{d}x\mathrm{d}y} \right)^{\frac{1}{p}}.
\end{equation*}
Furthermore, by saying that $v$ belongs to $W^{s,p}_{0}(\Omega)$, we mean that $v\in W^{s,p}\left( {\mathbb{G}^n} \right)$ and $v=0$ almosteverywhere in $\mathbb{G}^{n} \backslash \Omega$.

As outlined in the introduction,  the nonlocal tail of a function $v$ has been given in (\ref{nonlocal tail}) in the ball $B\left( {x_0,R} \right)$, which will play an important role in the rest of the paper.

We conclude this subsection by presenting the definition of the weak solution to a class of fractional problems under consideration. Let $\Omega$ be a bounded open set in $\mathbb{G}^n$ and $g\in W^{s,p}\left( {\mathbb{G}^n} \right)$, we focus  on weak solutions to the following integro-differential problems
\begin{equation}\label{integro-differential problem in sec2}
    \left \{
    \begin{aligned}
        &\mathcal{L}u=0&&\mathrm{in}\,\,\Omega,   \\
        &u=g &&\mathrm{in}\,\,{\mathbb{G}}^{n}\backslash\Omega,
    \end{aligned}
    \right.
\end{equation}
where the operator $\mathcal{L}$ is formally defined in \eqref{nonlocal operator L}. It should be noted that the boundary condition is defined over the entire complement of $\Omega$, as is customary when dealing with such nonlocal operators.

Now, let us consider the following functional
\begin{equation}\label{F function in sec2}
    \mathcal{F}(v)=\int_{{\mathbb{G}}^{n}} \int_{{\mathbb{G}}^{n}} {{\left| {d_{C}(x,y)} \right|}^{-Q-sp}{\left| {v(x)-v(y)} \right|}^{p}}\mathrm{d}x\mathrm{d}y
\end{equation}in $W^{s,p}\left( {\mathbb{G}^n} \right)$.
Theorem \ref{theo_Equivalence between weak solutions and the minimizers} below establishes that a unique $p\text{-}$minimizer of $\mathcal{F}$ exists for all $u\in W^{s,p}\left( {\mathbb{G}^n} \right)$ such that $u(x)=g(x)$ for $x\in{\mathbb{G}}^{n}\backslash\Omega$. Furthermore, a $p\text{-}$minimizer $u$ is a weak solution to problem \eqref{integro-differential problem in sec2} and vice versa. In order to specify the relevant spaces, for given $g\in W^{s,p}\left( {\mathbb{G}^n} \right)$, we define the convex sets of $W^{s,p}\left( {\mathbb{G}^n} \right)$ by
\begin{equation*}
    \mathcal{K}_{g}^{\pm}(\Omega):=\left\{ {v\in W^{s,p}\left( {\mathbb{G}^n} \right):(g-v)_{\pm}\in W^{s,p}_{0}\left( {\Omega} \right)} \right\}
\end{equation*}
and
\begin{equation*}
    \mathcal{K}_{g}(\Omega):=\mathcal{K}_{g}^{+}(\Omega)\cap \mathcal{K}_{g}^{-}(\Omega)=\left\{ {v\in W^{s,p}\left( {\mathbb{G}^n} \right):v-g\in W^{s,p}_{0}\left( {\Omega} \right)} \right\}.
\end{equation*}
Recall that the functions in the space $W^{s,p}_{0}\left( {\Omega} \right)$ are defined in the whole space, as they are extended to zero outside $\Omega$. Now we introduce the definition of weak subsolutions and supersolutions as well as weak solutions to problem \eqref{integro-differential problem in sec2} as follows.
\begin{definition}\label{def_weak solutions}
    Let $g\in W^{s,p}\left( {\mathbb{G}^n} \right)$. A function $u\in \mathcal{K}_{g}^{-}\left( {\mathcal{K}_{g}^{+}} \right)$ is a weak subsolution (supersolution) to \eqref{integro-differential problem in sec2} if
    \begin{equation}\label{test of sub(super)solution}
        \int_{\mathbb{G}^n}\int_{\mathbb{G}^n} {\left| {d_{C}(x,y)}\right|^{-Q-sp}\left| {u(x)-u(y)} \right|^{p-2}\left( {u(x)-u(y)} \right)\left( {\eta(x)-\eta(y)} \right)}\mathrm{d}x\mathrm{d}y\le(\ge)0
    \end{equation}
    for every nonnegative $\eta \in W^{s,p}_{0}(\Omega)$. A function $u$ is a weak solution to problem \eqref{integro-differential problem in sec2} if it is both weak sub- and supersolution. In particular, $u$ belongs to $\mathcal{K}_{g}(\Omega)$ and satisfies
    \begin{equation*}
        \int_{\mathbb{G}^n}\int_{\mathbb{G}^n} {\left| {d_{C}(x,y)}\right|^{-Q-sp}\left| {u(x)-u(y)} \right|^{p-2}\left( {u(x)-u(y)} \right)\left( {\eta(x)-\eta(y)} \right)}\mathrm{d}x\mathrm{d}y=0
    \end{equation*}
    for every $\eta \in W^{s,p}_{0}(\Omega)$.
\end{definition}

Similarly, we give the definition of sub- and superminimizers of
\eqref{F function in sec2} as follows.
\begin{definition}
    Let $g\in W^{s,p}\left( {\mathbb{G}^n} \right)$. A function $u\in \mathcal{K}_{g}^{-}$ is a subminimizer of the functional \eqref{F function in sec2} over $\mathcal{K}_{g}^{-}$ if $\mathcal{F}(u)\le \mathcal{F}(u+\eta)$ for every nonpositive $\eta \in W^{s,p}_{0}(\Omega)$. Similarly, a function $u\in \mathcal{K}_{g}^{+}$ is a superminimizer of the functional \eqref{F function in sec2} over $\mathcal{K}_{g}^{+}$ if $\mathcal{F}(u)\le \mathcal{F}(u+\eta)$ for every nonnegative $\eta \in W^{s,p}_{0}(\Omega)$. Finally, $u\in \mathcal{K}_{g}$ is a minimizer of the functional \eqref{F function in sec2} over $\mathcal{K}_{g}$ if $\mathcal{F}(u)\le \mathcal{F}(u+\eta)$ for every $\eta \in W^{s,p}_{0}(\Omega)$.
\end{definition}
\end{subsection}

\begin{subsection}{Equivalence between weak solutions and the minimizers}\label{subsection_Equivalence between weak solutions and the minimizers}

\begin{theorem}\label{theo_Equivalence between weak solutions and the minimizers}
    Let $s\in (0,1)$, $p\in\left[ 1,+\infty \right)$, and let $g\in W^{s,p}\left( {\mathbb{G}^n} \right)$. Then there exists a minimizer $u$ of \eqref{F function in sec2} over $\mathcal{K}_{g}$. Moreover, if $p>1$, then the solution is unique. Moreover, a function $u\in \mathcal{K}_{g}$ is a minimizer of \eqref{F function in sec2} over $\mathcal{K}_{g}$ if and only if it is a weak solution to problem \eqref{integro-differential problem in sec2}.
\end{theorem}
\begin{proof}
    We adopt the method of calculus of variations and the arguments in the first part of in \cite[Theorem 2.3]{di2016local} to prove the existence and uniqueness. We omit the details here.

    Moreover, it follows that $u$ solves the corresponding Euler-Lagrange equation by perturbing $u\in \mathcal{K}_{g}$ using standard perturbation techniques. Indeed, supposing that $u\in \mathcal{K}_{g}$ is a minimizer of \eqref{F function in sec2} over $\mathcal{K}_{g}$, take any $\phi \in W^{s,p}_{0}(\Omega)$ and calculate formally
    \begin{equation*}
        \begin{aligned}
            \frac{\mathrm{d}}{\mathrm{d}t}\mathcal{F}\left( {u+t\phi} \right)\bigg|_{t=0}&=\int_{\mathbb{G}^n}\int_{\mathbb{G}^n} {\left| {d_{C}(x,y)} \right|^{-Q-sp}\frac{\mathrm{d}}{\mathrm{d}t}\left| {u(x)-u(y)+t\left( {\phi(x)-\phi(y)} \right)} \right|^p}\mathrm{d}x\mathrm{d}y\Bigg|_{t=0}   \\
            &=\int_{\mathbb{G}^n}\int_{\mathbb{G}^n} {\left| {d_{C}(x,y)} \right|^{-Q-sp}\left| {u(x)-u(y)} \right|^{p-2}\left( {u(x)-u(y)} \right)\left( {\phi(x)-\phi(y)} \right)}\mathrm{d}x\mathrm{d}y   \\
            &=0.
        \end{aligned}
    \end{equation*}
    Since $u$ is a minimizer, the left-hand side vanishes, thus it  holds  that $u\in \mathcal{K}_{g}$ constitutes a weak solution to problem \eqref{integro-differential problem in sec2}. Conversely, let $u\in \mathcal{K}_{g}$ be a weak solution to \eqref{integro-differential problem in sec2} and consider $\phi=u-v\in W^{s,p}_{0}(\Omega)$, where $v\in \mathcal{K}_{g}$. Then, by Young's inequality,
    \begin{equation*}
        \begin{aligned}
            0=&\int_{\mathbb{G}^n}\int_{\mathbb{G}^n} {\left| {d_{C}(x,y)} \right|^{-Q-sp}\left| {u(x)-u(y)} \right|^{p-2}\left( {u(x)-u(y)} \right)\left( {\phi(x)-\phi(y)} \right)}\mathrm{d}x\mathrm{d}y   \\
            =&\int_{\mathbb{G}^n}\int_{\mathbb{G}^n} {\left| {d_{C}(x,y)} \right|^{-Q-sp}\left| {u(x)-u(y)} \right|^{p}}\mathrm{d}x\mathrm{d}y   \\
            &-\int_{\mathbb{G}^n}\int_{\mathbb{G}^n} {\left| {d_{C}(x,y)} \right|^{-Q-sp}\left| {u(x)-u(y)} \right|^{p-2}\left( {u(x)-u(y)} \right)\left( {v(x)-v(y)} \right)}\mathrm{d}x\mathrm{d}y   \\
            \ge&\int_{\mathbb{G}^n}\int_{\mathbb{G}^n} {\left| {d_{C}(x,y)} \right|^{-Q-sp}\left| {u(x)-u(y)} \right|^{p}}\mathrm{d}x\mathrm{d}y   \\
            &-\int_{\mathbb{G}^n}\int_{\mathbb{G}^n} {\left| {d_{C}(x,y)} \right|^{-Q-sp}\left( {\frac{\left| {v(x)-v(y)} \right|^{p}}{p}+\frac{\left| {u(x)-u(y)} \right|^{p}}{p/(p-1)}} \right)}\mathrm{d}x\mathrm{d}y   \\
            =&\frac{1}{p}\int_{\mathbb{G}^n}\int_{\mathbb{G}^n} {\left| {d_{C}(x,y)} \right|^{-Q-sp}\left| {u(x)-u(y)} \right|^{p}}\mathrm{d}x\mathrm{d}y   \\
            &-\frac{1}{p}\int_{\mathbb{G}^n}\int_{\mathbb{G}^n} {\left| {d_{C}(x,y)} \right|^{-Q-sp}\left| {v(x)-v(y)} \right|^{p}}\mathrm{d}x\mathrm{d}y,
        \end{aligned}
    \end{equation*}
    which implies that  $u$ is a minimizer of \eqref{F function in sec2} over $\mathcal{K}_{g}$.
\end{proof}
\end{subsection}
\end{section}

\begin{section}{Fundamental estimates}\label{section_Fundamental estimates}
In this section, we present the proofs of relevant estimates that will be utilized in subsequent analysis. We believe that these results may be of independent interest in the analysis of equations involving the (nonlinear) fractional Laplacian and related nonlocal operators. The first result provides a natural extension of the well-known Caccioppoli inequality to the nonlocal framework.

\textbf{Proof of Theorem \ref{theo_Caccioppoli estimates with tail}}.
For the sake of generality,  it should be noted  that the present proof is also valid when $p=1$. Consider $u$ as a weak solution. Testing \eqref{test of sub(super)solution} with $\eta:=w_{+}\phi^p$, where $\phi$ is any nonnegative function in $C_{0}^{\infty}\left( {B\left( {x_{0},r} \right)} \right)$, we have
\begin{equation}\label{Caccioppoli estimates_two terms}
    \begin{aligned}
        0\ge&\int_{\mathbb{G}^n}\int_{\mathbb{G}^n} {\left| {d_{C}(x,y)}\right|^{-Q-sp}\left| {u(x)-u(y)} \right|^{p-2}\left( {u(x)-u(y)} \right)\left( {\eta(x)-\eta(y)} \right)}\mathrm{d}x\mathrm{d}y   \\
        =&\int_{B_{r}}\int_{B_{r}} {\left| {d_{C}(x,y)}\right|^{-Q-sp}\left| {u(x)-u(y)} \right|^{p-2}\left( {u(x)-u(y)} \right)\left( {w_{+}(x)\phi^p(x)-w_{+}(y)\phi^p(y)} \right)}\mathrm{d}x\mathrm{d}y   \\
        &+2\int_{\mathbb{G}^n\backslash B_{r}}\int_{B_{r}} {\left| {d_{C}(x,y)}\right|^{-Q-sp}\left| {u(x)-u(y)} \right|^{p-2}\left( {u(x)-u(y)} \right)w_{+}(x)\phi^p(x)}\mathrm{d}x\mathrm{d}y.
    \end{aligned}
\end{equation}
Observe that $\eta$ qualifies as an admissible test function, since truncations of functions in $W^{s,p}\left( {{\mathbb{G}}^{n}} \right)$ remain in $W^{s,p}\left( {{\mathbb{G}}^{n}} \right)$.

We first deal with the integrands of the two terms in \eqref{Caccioppoli estimates_two terms} separately. For the first term, without loss of generality, we may assume  that $u(x)\ge u(y)$; otherwise just exchange the roles of $x$ and $y$ below.   Thus we otain
\begin{equation}\label{Caccioppoli estimates_first term}
    \begin{aligned}
        &\left| {u(x)-u(y)} \right|^{p-2}\left( {u(x)-u(y)} \right)\left( {w_{+}(x)\phi^p(x)-w_{+}(y)\phi^p(y)} \right)   \\
        =&\left( {u(x)-u(y)} \right)^{p-1}\left( {\left( {u(x)-k} \right)_{+}\phi^p(x)-\left( {u(y)-k} \right)_{+}\phi^p(y)} \right)   \\
        =&\left \{
        \begin{aligned}
            &\left( {w_{+}(x)-w_{+}(y)} \right)^{p-1}\left( {w_{+}(x)\phi^p(x)-w_{+}(y)\phi^p(y)} \right),&&u(x),u(y)>k   \\
            &\left( {u(x)-u(y)} \right)^{p-1}w_{+}(x)\phi^p(x),&&u(x)>k,u(y)\le k   \\
            &0,&&\mathrm{otherwise}
        \end{aligned}
        \right.   \\
        \ge&\left( {w_{+}(x)-w_{+}(y)} \right)^{p-1}\left( {w_{+}(x)\phi^p(x)-w_{+}(y)\phi^p(y)} \right).
    \end{aligned}
\end{equation}
Next we consider the second term on the right-hand side of the inequality in \eqref{Caccioppoli estimates_two terms}. At this time,  we obtain
\begin{equation*}
    \begin{aligned}
        \left| {u(x)-u(y)} \right|^{p-2}\left( {u(x)-u(y)} \right)w_{+}(x)&\ge-\left( {u(y)-u(x)} \right)_{+}^{p-1}(u(x)-k)_{+}   \\
        &\ge-\left( {u(y)-k} \right)_{+}^{p-1}(u(x)-k)_{+}   \\
        &=-w_{+}^{p-1}(y)w_{+}(x),
    \end{aligned}
\end{equation*}
and, upon further estimation, we obtain
\begin{equation}\label{Caccioppoli estimates_second term}
    \begin{aligned}
        &\int_{\mathbb{G}^n\backslash B_{r}}\int_{B_{r}} {\left| {d_{C}(x,y)}\right|^{-Q-sp}\left| {u(x)-u(y)} \right|^{p-2}\left( {u(x)-u(y)} \right)w_{+}(x)\phi^p(x)}\mathrm{d}x\mathrm{d}y   \\
        \ge&-\int_{\mathbb{G}^n\backslash B_{r}}\int_{B_{r}} {\left| {d_{C}(x,y)}\right|^{-Q-sp}w_{+}^{p-1}(y)w_{+}(x)\phi^p(x)}\mathrm{d}x\mathrm{d}y   \\
        \ge&-\int_{B_{r}} {w_{+}(x)\phi^p(x)}\mathrm{d}x\left( {\mathop{\sup}_{x\in \mathrm{supp}\,\phi}\int_{\mathbb{G}^n\backslash B_{r}} {\left| {d_{C}(x,y)}\right|^{-Q-sp}w_{+}^{p-1}(y)}\mathrm{d}y} \right).
    \end{aligned}
\end{equation}
By combining \eqref{Caccioppoli estimates_first term} and \eqref{Caccioppoli estimates_second term}, we thus deduce from \eqref{Caccioppoli estimates_two terms} that
\begin{equation}\label{Caccioppoli estimates_0>two terms replace}
    \begin{aligned}
        0\ge&\int_{B_{r}}\int_{B_{r}} \left| {d_{C}(x,y)}\right|^{-Q-sp}\left| {u(x)-u(y)} \right|^{p-2}   \\
        &\times \left( {w_{+}(x)-w_{+}(y)} \right)\left( {w_{+}(x)\phi^p(x)-w_{+}(y)\phi^p(y)} \right)\mathrm{d}x\mathrm{d}y   \\
        &-2\int_{B_{r}} {w_{+}(x)\phi^p(x)}\mathrm{d}x\left( {\mathop{\sup}_{x\in \mathrm{supp}\,\phi}\int_{\mathbb{G}^n\backslash B_{r}} {\left| {d_{C}(x,y)}\right|^{-Q-sp}w_{+}^{p-1}(y)}\mathrm{d}y} \right).
    \end{aligned}
\end{equation}
Furthermore, we examine the first term in the inequality \eqref{Caccioppoli estimates_0>two terms replace}. If $w_{+}(x)\ge w_{+}(y)$ and $\phi(x)\le\phi(y)$ in the integrand, we use Lemma \ref{lemma_Gamma small inequality}  to obtain
\begin{equation*}
    \phi^p(x)\ge\left( {1-c_{p}\epsilon} \right)\phi^p(y)-\left( {1+c_{p}\epsilon} \right)\epsilon^{1-p}\left| {\phi(x)-\phi(y)} \right|^p
\end{equation*}
for any $\epsilon \in\left( {0,1} \right]$ with the constant $c_{p}\equiv (p-1)\Gamma\left( {\max\{1,p-2\}} \right)$.
Thus, we choose
\begin{equation*}
    \epsilon :=\frac{1}{\max\left\{ {1,2c_{p}} \right\}}\frac{w_{+}(x)-w_{+}(y)}{w_{+}(x)}\in\left( {0,1} \right]
\end{equation*}
to obtain
\begin{equation*}
    \begin{aligned}
        \left( {w_{+}(x)-w_{+}(y)} \right)^{p-1}w_{+}(x)\phi^p(x)\ge&\left( {w_{+}(x)-w_{+}(y)} \right)^{p-1}w_{+}(x)\left( {\max\left\{ {\phi(x),\phi(y)} \right\}} \right)^p   \\
        &-\frac{1}{2}\left( {w_{+}(x)-w_{+}(y)} \right)^{p}\left( {\max\left\{ {\phi(x),\phi(y)} \right\}} \right)^p   \\
        &-c\left( {\max\left\{ {w_{+}(x),w_{+}(y)} \right\}} \right)^p\left| {\phi(x)-\phi(y)} \right|^p,
    \end{aligned}
\end{equation*}
where $c$ depends only on $p$. As noted in the estimate above, we assumed that $\phi(x)\le\phi(y)$ and $\max \left\{ {\phi(x),\phi(y)} \right\}=\phi(y)$. If $0=w_{+}(x)\ge w_{+}(y)\ge0$ or $w_{+}(x)\ge w_{+}(y)$ and $\phi(x)\ge\phi(y)$, the displayed estimate can be easily proved and    it holds in these cases as well. It follows that
\begin{equation*}
    \begin{aligned}
        &\left( {w_{+}(x)-w_{+}(y)} \right)^{p-1}\left( {w_{+}(x)\phi^p(x)-w_{+}(y)\phi^p(y)} \right)   \\
        \ge&\left( {w_{+}(x)-w_{+}(y)} \right)^{p-1}\left[ {w_{+}(x)\left( {\max\left\{ {\phi(x),\phi(y)} \right\}} \right)^p-w_{+}(y)\phi^p(y)} \right]   \\
        &-\frac{1}{2}\left( {w_{+}(x)-w_{+}(y)} \right)^{p}\left( {\max\left\{ {\phi(x),\phi(y)} \right\}} \right)^p   \\
        &-c\left( {\max\left\{ {w_{+}(x),w_{+}(y)} \right\}} \right)^p\left| {\phi(x)-\phi(y)} \right|^p   \\
        \ge&\frac{1}{2}\left( {w_{+}(x)-w_{+}(y)} \right)^{p}\left( {\max\left\{ {\phi(x),\phi(y)} \right\}} \right)^p   \\
        &-c\left( {\max\left\{ {w_{+}(x),w_{+}(y)} \right\}} \right)^p\left| {\phi(x)-\phi(y)}
        \right|^p
    \end{aligned}
\end{equation*}
whenever $w_{+}(x)\ge w_{+}(y)$. Conversely, if $w_{+}(y)> w_{+}(x)$ in the integrand, we change the roles of $x$ and $y$   in the displayed expression by similar arguments. Thus, in all cases, we obtain
\begin{equation}\label{Caccioppoli estimates_first term replace eqref(Caccioppoli estimates_0>two terms replace)}
    \begin{aligned}
        &\int_{B_{r}}\int_{B_{r}} \left| {d_{C}(x,y)}\right|^{-Q-sp}\left| {u(x)-u(y)} \right|^{p-2}   \\
        &\times \left( {w_{+}(x)-w_{+}(y)} \right)\left( {w_{+}(x)\phi^p(x)-w_{+}(y)\phi^p(y)} \right)\mathrm{d}x\mathrm{d}y   \\
        \ge&\frac{1}{2}\int_{B_{r}}\int_{B_{r}} {\left| {d_{C}(x,y)} \right|^{-Q-sp}\left| {w_{+}(x)-w_{+}(y)} \right|^{p}\left( {\max\left\{ {\phi(x),\phi(y)} \right\}} \right)^p}\mathrm{d}x\mathrm{d}y   \\
        &-c\int_{B_{r}}\int_{B_{r}} {\left| {d_{C}(x,y)} \right|^{-Q-sp}\left( {\max\left\{ {w_{+}(x),w_{+}(y)} \right\}} \right)^p\left| {\phi(x)-\phi(y)} \right|^p}\mathrm{d}x\mathrm{d}y.
    \end{aligned}
\end{equation}
Now, we note that
\begin{equation*}
    \begin{aligned}
        \left| {w_{+}(x)\phi(x)-w_{+}(y)\phi(y)} \right|^p\le&2^{p-1}\left| {w_{+}(x)-w_{+}(y)} \right|^p\left( {\max\left\{ {\phi(x),\phi(y)} \right\}} \right)^p   \\
        &+2^{p-1}\left( {\max\left\{ {w_{+}(x),w_{+}(y)} \right\}} \right)^p\left| {\phi(x)-\phi(y)}
        \right|^p.
    \end{aligned}
\end{equation*}
We combine the preceding inequality with \eqref{Caccioppoli estimates_0>two terms replace} and  \eqref{Caccioppoli estimates_first term replace eqref(Caccioppoli estimates_0>two terms replace)} to deduce the proof of \eqref{Caccioppoli estimates with tail} for $w_{+}$.

To establish the estimate in \eqref{Caccioppoli estimates with tail} for $w_{-}$, it is sufficient to use the same arguments as above and use  the function $\eta=-w_{-}\phi$ instead of $\eta=w_{+}\phi$ as a test function in the weak formulation of problem \eqref{integro-differential problem in sec2}.

In the  above proof, we utilized the following technical inequality(see \cite[Lemma 3.1]{di2016local}), which will also be useful in subsequent proofs.
\begin{lemma}\label{lemma_Gamma small inequality}
    Let $p\ge1$ and $\epsilon \in\left( {0,1} \right]$. Then
    \begin{equation*}
        \left| a \right|^p\le\left| b \right|^p +c_{p}\epsilon\left| b \right|^p+\left( {1+c_{p}\epsilon} \right)\epsilon^{1-p}\left| a-b \right|^p, \quad c_{p}:=(p-1)\Gamma\left( {\max\{1,p-2\}} \right),
    \end{equation*}
    holds for every $a,b\in \mathbb{R}^l$ and $l\ge1$, where $\Gamma$ stands for the standard Gamma function.
\end{lemma}

In what follows, we prove Logarithmic Lemma, which is the second main tool in this paper.
\begin{proof}[\bf Proof of Theorem \ref*{lemma_Logarithmic Lemma}]
    Let $d>0$ be a parameter. It follows from \cite[Theorem 3.3]{Danielli} that we can choose a smooth cut-off function $\phi \in C_{0}^{\infty}\left( {B_{3r/2}} \right)$ such that
    \begin{equation*}
        0\le\phi\le1,\quad \phi\equiv1 \mathrm  \,\, \mathrm{in} \,\, B_{r} \,\,\mathrm{and}\,\, \left| {\nabla_{\mathbb{G}^{n}}\phi} \right|<cr^{-1}\,\, \mathrm{in}\,\,B_{r}\subset B_{R/2}.
    \end{equation*}
    In the weak formulation of \eqref{integro-differential problem in sec2}, we utilize the test function $\eta$ in the Definition \ref{def_weak solutions}, defined by
    \begin{equation*}
        \eta =\left( {u+d} \right)^{1-p}\phi^{p}.
    \end{equation*}
    Then we have
    \begin{equation}\label{0=I1+I2 in proof of Logarithmic Lemma}
        \begin{aligned}
            0=&\int_{\mathbb{G}^n}\int_{\mathbb{G}^n} {\left| {d_{C}(x,y)}\right|^{-Q-sp}\left| {u(x)-u(y)} \right|^{p-2}\left( {u(x)-u(y)} \right)\left( {\eta(x)-\eta(y)} \right)}\mathrm{d}x\mathrm{d}y   \\
            =&\int_{B_{2r}}\int_{B_{2r}} {\left| {d_{C}(x,y)}\right|^{-Q-sp}\left| {u(x)-u(y)} \right|^{p-2}\left( {u(x)-u(y)} \right)}   \\
            &\times \left[ {\frac{\phi^{p}(x)}{\left( {u(x)+d} \right)^{p-1}}-\frac{\phi^{p}(y)}{\left( {u(y)+d} \right)^{p-1}}} \right]\mathrm{d}x\mathrm{d}y   \\
            &+2\int_{\mathbb{G}^n\backslash B_{2r}}\int_{B_{2r}} {\left| {d_{C}(x,y)}\right|^{-Q-sp}\left| {u(x)-u(y)} \right|^{p-2}\frac{u(x)-u(y)}{\left( {u(x)+d} \right)^{p-1}}\phi^{p}(x)}\mathrm{d}x\mathrm{d}y   \\
            =&:I_{1}+I_{2}
        \end{aligned}
    \end{equation}
    We first consider the term $I_1$.  If $u(x)\ge u(y)$, we use inequality in Lemma \ref{lemma_Gamma small inequality} via choosing $a=\phi(x)$ and $b=\phi(y)$, and
    \begin{equation*}
        \epsilon =\delta \frac{u(x)-u(y)}{u(x)+d}, \quad \delta \in(0,1),
    \end{equation*}
    since $u(y)\ge0$ for any $y\in B_{2r}\subset B_{R}$. Therefore, we  estimate the integrand in $I_{1}$ as follows,
    \begin{equation}\label{integral of I1 estimate in proof of Logarithmic Lemma}
        \begin{aligned}
            &\left| {d_{C}(x,y)}\right|^{-Q-sp}\left| {u(x)-u(y)} \right|^{p-2}\left( {u(x)-u(y)} \right)\left[ {\frac{\phi^{p}(x)}{\left( {u(x)+d} \right)^{p-1}}-\frac{\phi^{p}(y)}{\left( {u(y)+d} \right)^{p-1}}} \right]   \\
            \le&\left| {d_{C}(x,y)}\right|^{-Q-sp}\frac{\left( {u(x)-u(y)} \right)^{p-1}}{\left( {u(x)+d} \right)^{p-1}}\phi^{p}(y)\left[ {1+c\delta\frac{u(x)-u(y)}{u(x)+d}-\left( {\frac{u(x)+d}{u(y)+d}} \right)^{p-1}} \right]   \\
            &+c\delta^{1-p}\left| {d_{C}(x,y)}\right|^{-Q-sp}\left| {\phi(x)-\phi(y)} \right|^p,
        \end{aligned}
    \end{equation}
    where $c\equiv c(p)$. Note that the first term on the right-hand side of (\ref{integral of I1 estimate in proof of Logarithmic Lemma}) can be rewritten as
    \begin{equation}\label{J1 in proof of Logarithmic Lemma}
        \left| {d_{C}(x,y)}\right|^{-Q-sp}\left( {\frac{u(x)-u(y)}{u(x)+d}} \right)^{p}\phi^{p}(y)\left[ {\frac{1-\left( {\frac{u(y)+d}{u(x)+d}} \right)^{1-p}}{1-\frac{u(y)+d}{u(x)+d}}+c\delta} \right]=:J_{1}.
    \end{equation}

    Define the real function $g(t)$ as
    \begin{equation*}
        g(t):=\frac{1-t^{1-p}}{1-t}=-\frac{p-1}{1-t}\int_{t}^{1} {\tau^{-p}}\mathrm{d}\tau \quad \forall t\in(0,1).
    \end{equation*}
    It is easy to see that  the function $g$ is increasing in $t$ and
    \begin{equation*}
        g(t)\le-(p-1),\quad \forall \,t\in(0,1);
    \end{equation*}
    moreover, for any $t\le1/2$,
    \begin{equation*}
        g(t)\le -\frac{p-1}{2^p}\frac{t^{1-p}}{1-t}.
    \end{equation*}
    Thus, if
    \begin{equation*}
        t=\frac{u(y)+d}{u(x)+d}\in ( {0,\frac{1}{2}}  ],
    \end{equation*}
    that is
    \begin{equation*}
        u(y)+d\le\frac{u(x)+d}{2},
    \end{equation*}
    then  we get
    \begin{equation*}
        J_{1}\le \left| {d_{C}(x,y)}\right|^{-Q-sp}\left( {c\delta-\frac{p-1}{2^p}} \right)\left[ {\frac{u(x)-u(y)}{u(y)+d}}
        \right]^{p-1}\phi^p(y)
    \end{equation*} due to  $\left( {u(x)-u(y)} \right)\left( {u(y)+d} \right)^{p-1}/\left( {u(x)+d}
    \right)^{p}\le1$.
    Hence, we select the appropriate $\delta$ as
    \begin{equation}\label{delta in proof of Logarithmic Lemma}
        \delta=\frac{p-1}{2^{p+1}c}
    \end{equation}
    in the preceding inequality to obtain
    \begin{equation}\label{estimate s1 of J1 in proof of Logarithmic Lemma}
        J_{1}\le -\left| {d_{C}(x,y)}\right|^{-Q-sp}\frac{p-1}{2^{p+1}}\left[ {\frac{u(x)-u(y)}{u(y)+d}} \right]^{p-1}.
    \end{equation}
    Now we consider the case
    \begin{equation*}
        u(y)+d>\frac{u(x)+d}{2},
    \end{equation*}
    i.e.,
    \begin{equation*}
        t=\frac{u(y)+d}{u(x)+d}\in ( {\frac{1}{2},1}  ),
    \end{equation*}
    then
    \begin{equation*}
        J_{1}\le \left| {d_{C}(x,y)}\right|^{-Q-sp}\left( {c\delta-(p-1)} \right)\left[ {\frac{u(x)-u(y)}{u(y)+d}} \right]^{p}\phi^p(y).
    \end{equation*}
    The parameter $\delta$ can be chosen as in \eqref{delta in proof of Logarithmic Lemma}, and we have
    \begin{equation}\label{estimate s2 of J1 in proof of Logarithmic Lemma}
        J_{1}\le -\left| {d_{C}(x,y)}\right|^{-Q-sp}\frac{\left( {2^{p+1}-1} \right)\left( {p-1} \right)}{2^{p+1}}\left[ {\frac{u(x)-u(y)}{u(y)+d}} \right]^{p}\phi^p(y).
    \end{equation}
    Furthermore, if $2(u(y)+d)< u(x)+d$, then the following inequality holds with $c\equiv c(p)$:
    \begin{equation}\label{log inequality s1 in proof of Logarithmic Lemma}
        \left[ {\mathrm{log}\left( {\frac{u(x)+d}{u(y)+d}} \right)} \right]^{p}\le c\left[ {\frac{u(x)-u(y)}{u(y)+d}} \right]^{p-1},
    \end{equation}
    where we have utilized the fact that $\left( {\mathrm{log} \,\xi} \right)^p\le c\left( {\xi-1} \right)^{p-1}$ when $\xi>2$. On the other hand, if $2(u(y)+d)\ge u(x)+d$, noting that we have assumed $u(x)>u(y)$, then
    \begin{equation}\label{log inequality s2 in proof of Logarithmic Lemma}
        \left[ {\mathrm{log}\left( {\frac{u(x)+d}{u(y)+d}} \right)} \right]^{p}= \left[ {\mathrm{log}\left( {1+\frac{u(x)-u(y)}{u(y)+d}} \right)} \right]^{p}\le2^p\left( {\frac{u(x)-u(y)}{u(x)+d}} \right)^p
    \end{equation}
    where we have used
    \begin{equation*}
        \mathrm{log}\left( {1+\xi} \right)<\xi,\quad \forall\,\xi>0,\quad \mathrm{with}\;\,\xi=\frac{u(x)-u(y)}{u(y)+d}\le\frac{2[u(x)-u(y)]}{u(x)+d}.
    \end{equation*}

    Thus, combining \eqref{integral of I1 estimate in proof of Logarithmic Lemma} with \eqref{J1 in proof of Logarithmic Lemma}, \eqref{estimate s1 of J1 in proof of Logarithmic Lemma}, \eqref{estimate s2 of J1 in proof of Logarithmic Lemma}, \eqref{log inequality s1 in proof of Logarithmic Lemma} and \eqref{log inequality s2 in proof of Logarithmic Lemma}, we conclude with
    \begin{equation*}
        \begin{aligned}
            &\left| {d_{C}(x,y)}\right|^{-Q-sp}\left| {u(x)-u(y)} \right|^{p-2}\left( {u(x)-u(y)} \right)\left[ {\frac{\phi^{p}(x)}{\left( {u(x)+d} \right)^{p-1}}-\frac{\phi^{p}(y)}{\left( {u(y)+d} \right)^{p-1}}} \right]   \\
            \le&-\frac{1}{c}\left| {d_{C}(x,y)}\right|^{-Q-sp}\left[ {\mathrm{log}\left( {\frac{u(x)+d}{u(y)+d}} \right)} \right]^p\phi^{p}(y)   \\
            &+c\delta^{1-p}\left| {d_{C}(x,y)}\right|^{-Q-sp}\left| {\phi(x)-\phi(y)} \right|^p.
        \end{aligned}
    \end{equation*}
    Observe that if $u(x)=u(y)$, the same estimate above holds trivially. If, on the other hand, $u(y)>u(x)$, we can interchange the roles of $x$ and $y$ in the computation above. Thus, we obtain the estimate for the integral $I_1$ in \eqref{0=I1+I2 in proof of Logarithmic Lemma} as
    \begin{equation}\label{I1 estimate in proof of Logarithmic Lemma}
        \begin{aligned}
            I_{1}\le& -\frac{1}{c}\int_{B_{2r}}\int_{B_{2r}} {\left| {d_{C}(x,y)}\right|^{-Q-sp}\left| {\mathrm{log}\left( {\frac{u(x)+d}{u(y)+d}} \right)} \right|^p\phi^{p}(y)}\mathrm{d}x\mathrm{d}y   \\
            &+c\int_{B_{2r}}\int_{B_{2r}} {\left| {d_{C}(x,y)}\right|^{-Q-sp}\left| {\phi(x)-\phi(y)} \right|^p}\mathrm{d}x\mathrm{d}y
        \end{aligned}
    \end{equation}
    for a constant $c\equiv c(p)$ by the choice of $\delta$.

    In what follows, we consider  $I_{2}$ in \eqref{0=I1+I2 in proof of Logarithmic Lemma}. First, observe that when $y\in B_{R}$, $u(y)\ge0$, thus
    \begin{equation*}
        \frac{\left( {u(x)-u(y)} \right)_{+}^{p-1}}{\left( {d+u(x)} \right)^{p-1}}\le1 \quad\forall \,x\in B_{2r},\,\,y\in B_{R}.
    \end{equation*}
    Moreover, when $y\in\mathbb{G}^n\backslash B_{R}$,
    \begin{equation*}
        \left( {u(x)-u(y)} \right)_{+}^{p-1}\le2^{p-1}\left[ {u^{p-1}(x)+\left( {u(y)} \right)_{-}^{p-1}} \right]  \quad\forall\,x\in B_{2r}.
    \end{equation*}
    Therefore,
    \begin{equation*}
        \begin{aligned}
            I_2\le&2\int_{B_R\backslash B_{2r}}\int_{B_{2r}} {\left| {d_{C}(x,y)}\right|^{-Q-sp}\frac{\left( {u(x)-u(y)} \right)_{+}^{p-1}}{\left( {u(x)+d} \right)^{p-1}}\phi^{p}(x)}\mathrm{d}x\mathrm{d}y   \\
            &+2\int_{\mathbb{G}^n\backslash B_{R}}\int_{B_{2r}} {\left| {d_{C}(x,y)}\right|^{-Q-sp}\frac{\left( {u(x)-u(y)} \right)_{+}^{p-1}}{\left( {u(x)+d} \right)^{p-1}}\phi^{p}(x)}\mathrm{d}x\mathrm{d}y   \\
            \le&c\int_{B_R\backslash B_{2r}}\int_{B_{2r}} {\left| {d_{C}(x,y)}\right|^{-Q-sp}\phi^{p}(x)}\mathrm{d}x\mathrm{d}y   \\
            &+c\int_{\mathbb{G}^n\backslash B_{R}}\int_{B_{2r}} {\left| {d_{C}(x,y)}\right|^{-Q-sp}\frac{u^{p-1}(x)+\left( {u(y)} \right)_{-}^{p-1}}{\left( {u(x)+d} \right)^{p-1}}\phi^{p}(x)}\mathrm{d}x\mathrm{d}y.
        \end{aligned}
    \end{equation*}
    Then we estimate the integral $I_2$ as
    \begin{equation}\label{I2 le I2,1+I2,2 in proof of Logarithmic Lemma}
        \begin{aligned}
            I_2\le&c\int_{B_R\backslash B_{2r}}\int_{B_{2r}} {\left| {d_{C}(x,y)}\right|^{-Q-sp}\phi^{p}(x)}\mathrm{d}x\mathrm{d}y   \\
            &+c\int_{\mathbb{G}^n\backslash B_{R}}\int_{B_{2r}} {\left| {d_{C}(x,y)}\right|^{-Q-sp}\phi^{p}(x)}\mathrm{d}x\mathrm{d}y   \\
            &+cd^{1-p}\int_{\mathbb{G}^n\backslash B_{R}}\int_{B_{2r}} {\left| {d_{C}(x,y)}\right|^{-Q-sp}\left( {u(y)} \right)_{-}^{p-1}}\mathrm{d}x\mathrm{d}y   \\
            =&c\int_{\mathbb{G}^n\backslash B_{2r}}\int_{B_{2r}} {\left| {d_{C}(x,y)}\right|^{-Q-sp}\phi^{p}(x)}\mathrm{d}x\mathrm{d}y   \\
            &+cd^{1-p}\int_{\mathbb{G}^n\backslash B_{R}}\int_{B_{2r}} {\left| {d_{C}(x,y)}\right|^{-Q-sp}\left( {u(y)} \right)_{-}^{p-1}}\mathrm{d}x\mathrm{d}y   \\
            :=&I_{2,1}+I_{2,2},
        \end{aligned}
    \end{equation}
    where $c\equiv c(p)$, since $u(x)\ge0$ and $\phi(x)\le1$ on $B_{2r}$.

    From now on, the proof diverges from the Euclidean case presented in \cite{di2016local}, where the logarithmic estimates plainly follows, since the measure of the ball in the Grushin-type spaces depends on its center by Lemma \ref{lemma_measure of ball in Grushin spaces}, unlike the case in the Euclidean space.

    For $I_{2,2}$, observe that for any $x\in B_{r}$, $y\in \mathbb{G}^n\backslash B_{R}$ and $2r\le R$,
    \begin{equation*}
        \frac{\left| {d_C\left( {x_0,y} \right)} \right|}{\left| {d_C\left( {x,y} \right)} \right|}\le1+\frac{\left| {d_C\left( {x_0,x} \right)} \right|}{\left| {d_C\left( {x,y} \right)} \right|}\le1+\frac{r}{R-r}\le2,
    \end{equation*}
    and thus, we obtain
    \begin{equation}\label{I2,2 estimate in proof of Logarithmic Lemma}
        \begin{aligned}
            I_{2,2}&\le cd^{1-p}\left| {B(x_0,r)} \right|\int_{\mathbb{G}^n\backslash B_{R}} {\left| {d_{C}(x_0,y)}\right|^{-Q-sp}\left( {u(y)} \right)_{-}^{p-1}}\mathrm{d}y   \\
            &\le cd^{1-p}\left| {B(x_0,r)} \right|R^{-sp}\left[ {\mathrm{Tail}\left( {u_{-};x_{0},R} \right)} \right]^{p-1}   \\
            &\le cd^{1-p}\frac{r^n\left( {\left| {x_0} \right|_h+2r} \right)^{m(n-h)}}{R^{sp}}
            \left[ {\mathrm{Tail}\left( {u_{-};x_{0},R} \right)}
            \right]^{p-1},
        \end{aligned}
    \end{equation}
    where Lemma \ref{lemma_measure of ball in Grushin spaces} was applied with the constant $c$ depending only on $n,h,p,s$ and $m$.

    Note that for $x\in B ( {x_{0},3r/2})$ and $y\in \mathbb{G}^n\backslash B( {x_0, 2r})$, $$d_C(x_0,y)>2r>\frac{4}{3}d_C(x_0,x) $$ and $$d_C(x,y)\ge d_C(x_0,y)-d_C(x_0,x)\ge \frac{1}{4}d_C(x_0,y).$$ The integral $I_{2,1}$ can be estimated by recalling the definition of the cut-off function $\phi$ as
    \begin{equation}\label{I2,1 estimate in proof of Logarithmic Lemma}
        \begin{aligned}
            I_{2,1}&\le c\left| {B\left( {x_0,2r} \right)} \right|\mathop{\sup}_{x\in B\left( {x_{0},3r/2} \right)}\int_{\mathbb{G}^n\backslash B_{2r}} {\left| {d_{C}(x,y)}\right|^{-Q-sp}}\mathrm{d}y   \\
            &\le c\left| {B\left( {x_0,2r} \right)} \right|\mathop{\sup}_{x\in B\left( {x_{0},3r/2} \right)}\sum_{j=0}^{\infty}\int_{\left\{ {y\big|2^{j+2}r>d_{C}\left( {x_0,y} \right)\ge2^{j+1}r} \right\}}{\left| {d_{C}(x,y)}\right|^{-Q-sp}}\mathrm{d}y   \\
            &\le c | {B\left( {x_0,2r} \right)}| \sum_{j=0}^{\infty}\int_{\left\{ {y\big|2^{j+2}r>d_{C}\left( {x_0,y} \right)\ge2^{j+1}r} \right\}}{\left| {d_{C}(x_0,y)}\right|^{-Q-sp}}\mathrm{d}y   \\
            &\le cr^n\left( {\left| {x_0} \right|_h+2r} \right)^{m(n-h)}\sum_{j=0}^{\infty}\left[   {2^{j+1}   r} \right]^{-Q-sp}\left| {B\left( {x_0,2^{j+2}r} \right)} \right|   \\
            &\le cr^n\left( {\left| {x_0} \right|_h+2r} \right)^{m(n-h)}\sum_{j=0}^{\infty}\left[   {2^{j+1}   r} \right]^{-Q-sp}\left( {2^{j+2}r} \right)^n\left( {\left| {x_0} \right|_h+2^{j+2}r} \right)^{m(n-h)}   \\
            &\le cr^{2n-Q-sp}\left( {\left| {x_0} \right|_h+2r} \right)^{m(n-h)}\sum_{j=0}^{\infty}\left[   {2^{j+1}   r} \right]^{-Q-sp}\left( {2^{j+2}} \right)^n\left[ {\left| {x_0} \right|_h^{m(n-h)}+\left( {2^{j+2}r} \right)^{m(n-h)}} \right]   \\
            &\le cr^{2n-Q-sp}\left( {\left| {x_0} \right|_h+2r} \right)^{m(n-h)}\left[ {\sum_{j=0}^{\infty}\left( {2^j} \right)^{n-Q-sp}\left| {x_0} \right|_h^{m(n-h)}+\sum_{j=0}^{\infty}\left( {2^j} \right)^{-sp}r^{m(n-h)}} \right]   \\
            &\le cr^{2n-Q-sp}\left( {\left| {x_0} \right|_h+2r} \right)^{m(n-h)}\left[ {\left| {x_0} \right|_h^{m(n-h)}+r^{m(n-h)}} \right]   \\
            &\le cr^{2n-Q-sp}\left( {\left| {x_0} \right|+2r}
            \right)^{2m(n-h)},
        \end{aligned}
    \end{equation}
    where $c\equiv c(n,h,p,s,m)$. By integrating equations \eqref{0=I1+I2 in proof of Logarithmic Lemma}, \eqref{I1 estimate in proof of Logarithmic Lemma}, \eqref{I2 le I2,1+I2,2 in proof of Logarithmic Lemma}, \eqref{I2,2 estimate in proof of Logarithmic Lemma} and \eqref{I2,1 estimate in proof of Logarithmic Lemma}, the following result is derived,
    \begin{equation}\label{complex eq1 in proof of Logarithmic Lemma}
        \begin{aligned}
            &\int_{B_{2r}}\int_{B_{2r}} {\left| {d_{C}(x,y)}\right|^{-Q-sp}\left| {\mathrm{log}\left( {\frac{u(x)+d}{u(y)+d}} \right)} \right|^p\phi^{p}(y)}\mathrm{d}x\mathrm{d}y   \\
            \le&c\int_{B_{2r}}\int_{B_{2r}} {\left| {d_{C}(x,y)}\right|^{-Q-sp}\left| {\phi(x)-\phi(y)} \right|^p}\mathrm{d}x\mathrm{d}y   \\
            &+cr^{2n-Q-sp}\left( {\left| {x_0} \right|_h+2r} \right)^{2m(n-h)}   \\
            &+cd^{1-p}r^n\left( {\left| {x_0} \right|_h+2r} \right)^{m(n-h)}R^{-sp}\left[ {\mathrm{Tail}\left( {u_{-};x_{0},R} \right)}
            \right]^{p-1}.
        \end{aligned}
    \end{equation}
    In the subsequent calculation, the following mean value inequality holds: for any $x,y\in B(x_0,3r/2)$,
    \begin{equation}\label{3.19}
        \left| {\phi(x)-\phi(y)} \right|\le c\mathop{\sup}_{B\left( {x_{0},3r/2} \right)}\left| {\nabla _{\mathbb{G}^{n}}\phi} \right|
        {d_C(x,y)},
    \end{equation}
    where $c$ depends only on $m$.  In fact, let $B_i(x)=B(x,2^{-i}d_C(x,y))$ for each nonnegative integer $i$. Then $(\phi)_{B_i(x)}\rightarrow \phi(x)$ as $i\rightarrow \infty$. Using the triangle inequality, Lemma \ref{lemma_measure of ball in Grushin spaces}, the H\"{o}lder inequality and the Sobolev-Poincar\'{e}  inequality (see \cite[Theorem I]{Franchi1}), we have
    \begin{equation*}
        \begin{aligned}
            &|\phi(x)-(\phi)_{B_0(x)}|   \\
            \le&\sum^\infty_{i=0}|(\phi)_{B_i(x)}-(\phi)_{B_{i+1}(x)}|    \\
            \le&\sum^\infty_{i=0}\dashint_{B_{i+1}(x)}  | \phi  (z) -(\phi)_{B_i(x)}|\mathrm{d}z   \\
            \le&c\sum^\infty_{i=0}\dashint_{B_{i }(x)}  | \phi (z) -(\phi)_{B_i(x)}|\mathrm{d}z   \\
            \le&c\sum^\infty_{i=0}\Big(\dashint_{B_{i }(x)}  | \phi (z) -(\phi)_{B_i(x)}|^q\mathrm{d}z\Big)^{\frac{1}{q}}   \\
            \le&c\sum^\infty_{i=0}2^{-i }d_C(x,y)\Big(\dashint_{B_{i }(x)}  \left| {\nabla _{\mathbb{G}^{n}}\phi(z)} \right|^p\mathrm{d}z\Big)^{\frac{1}{p}}   \\
            \le&c \mathop{\sup}_{B\left( {x_{0},3r/2} \right)}\left| {\nabla_{\mathbb{G}^{n}}\phi} \right| {d_C(x,y)}
        \end{aligned}
    \end{equation*} with $1\le q\le p<\infty.$
    We also use the similar arguments to obtain
    \begin{equation*}\begin{aligned}|\phi(x)-(\phi)_{B_0(y)}|&\le   c \mathop{\sup}_{B\left( {x_{0},3r/2} \right)}\left| {\nabla
                _{\mathbb{G}^{n}}\phi} \right|
            {d_C(x,y)}.
        \end{aligned}
    \end{equation*}
    Furthermore, we have
    \begin{equation*}
        \begin{aligned}
            |(\phi)_{B_0(x)}-(\phi)_{B_0(y)}|&\le |(\phi)_{B_0(x)}-(\phi)_{2B_0(x)}|+|(\phi)_{2B_0(x)}-(\phi)_{B_0(y)}|\\
            &\le c \dashint_{2B_{0}(x)}  | \phi (z) -(\phi)_{2B_0(x)}|\mathrm{d}z\\
            &\le c d_C(x,y)\Big(\dashint_{2B_{0}(x)}  \left| {\nabla _{\mathbb{G}^{n}}\phi(z)} \right|^p\mathrm{d}z\Big)^{\frac{1}{p}}\\
            &\le c \mathop{\sup}_{B\left( {x_{0},3r/2} \right)}\left| {\nabla_{\mathbb{G}^{n}}\phi} \right| {d_C(x,y)}.
        \end{aligned}
    \end{equation*}
    Therefore, (\ref{3.19}) can be deduced by combining the above three inequalities.

    Note that for any  $x\in B(x_0,2r)$,  the triangle inequality implies that $B(x_0,2r)\subseteq B(x,4r)$. Then by the estimate on $\left| {\nabla_{\mathbb{G}^{n}}\phi} \right|$ in (\ref{3.19}), we obtain
    \begin{align}\label{complex eq2 in proof of Logarithmic Lemma}
        &\int_{B_{2r}}\int_{B_{2r}} {\left| {d_{C}(x,y)}\right|^{-Q-sp}\left| {\phi(x)-\phi(y)} \right|^p}\mathrm{d}x\mathrm{d}y   \notag\\
        \le&cr^{-p}\int_{B_{2r}}\int_{B (x,{4r})} {\left| {d_{C}(x,y)}\right|^{-Q-sp+p}}\mathrm{d}x\mathrm{d}y   \notag\\
        \le&cr^{-p}\int_{B_{2r}}\sum_{j=0}^{\infty}\int_{\left\{ {x\big|2^{-j+1}r>d_{C}\left( {x,y} \right)\ge2^{-j}r} \right\}}
        {\left| {d_{C}(x,y)}\right|^{-Q-sp+p}}\mathrm{d}y\mathrm{d}x   \notag\\
        \le&cr^{-p}\int_{B_{2r}}\sum_{j=0}^{\infty} {\left( {2^{-j}r} \right)^{-Q-sp+p}\left| {B\left( {x,2^{-j}r} \right)} \right|}\mathrm{d}y   \notag\\
        \le&cr^{-p}\left| {B\left( {x_0,2r} \right)} \right|\sum_{j=0}^{\infty} \left( {2^{-j}r} \right)^{-Q-sp+p}\left| {B\left( {x_0,2^{-j}r} \right)} \right|   \notag\\
        \le&cr^{2n-Q-sp}\left( {\left| {x_0} \right|_h+2r} \right)^{2m(n-h)}\sum_{j=0}^{\infty}\left( {2^{-j}} \right)^{-Q+n-sp+p}   \notag\\
        \le&cr^{2n-Q-sp}\left( {\left| {x_0} \right|_h+2r}
        \right)^{2m(n-h)},
    \end{align}
    where $c\equiv c(n,h,p,s,m)$ and $-Q+n-sp+p>0$, which means $p>m(n-h)$ and $0<s<1-m(n-h)/p$. By integrating \eqref{complex eq1 in proof of Logarithmic Lemma} and \eqref{complex eq2 in proof of Logarithmic Lemma} with the previously stated conditions on $p$ and $s$, we establish that Theorem \ref{lemma_Logarithmic Lemma} holds for $p>\max\left\{ {m(n-h),1} \right\}$, thereby concluding the proof.
\end{proof}

Motivated by Lemma 2.2 in \cite{dyda2023fractional}, we provide the fractional Poincar{\'e} type inequality for Grushin-type spaces, which will play a crucial role in subsequent proofs.
\begin{lemma}\label{lemma_Poincare inequality}
    Let $p>1$ and $s \in (0,1)$. For a ball $B_r\equiv B\left( {x_0,r} \right)\subset\Omega\subset\mathbb{G}^n$
    and any $v\in C^\infty_0(B\left( {x_0,r} \right))$,    we have
    \begin{equation}\label{Poincare inequality}
        \begin{aligned}
            &\dashint_{B_r} {\left| {v-(v)_{B_{r}}} \right|^p}\mathrm{d}x   \\
            \le&cr^{-2n+Q+sp}\left( {\left| {x_0} \right|_h+r} \right)^{-2m(n-h)}\int_{B_{r}}\int_{B_{r}} {\left| {d_{C}(x,y)}
                \right|^{-Q-sp}\left| {v(x)-v(y)}
                \right|^p}\mathrm{d}x\mathrm{d}y,
        \end{aligned}
    \end{equation}
    where $c$ depends only on $n,h,p,s$ and $m$.
\end{lemma}
\begin{proof}
    Using Lemma \ref{lemma_measure of ball in Grushin spaces} and  Jensen's inequality, we obtain
    \begin{equation*}
        \begin{aligned}
            &\dashint_{B_r} {\left| {v-(v)_{B_{r}}} \right|^p}\mathrm{d}x   \\
            \le&\dashint_{B_r}\dashint_{B_r} {\left| {v(x)-v(y)} \right|^p}\mathrm{d}x\mathrm{d}y   \\
            \le&cr^{Q+sp}\dashint_{B_r}\dashint_{B_r} {\left| {d_{C}(x,y)}\right|^{-Q-sp}\left| {v(x)-v(y)} \right|^p}\mathrm{d}x\mathrm{d}y   \\
            \le&cr^{-2n+Q+sp}\left( {\left| {x_0} \right|_h+r} \right)^{-2m(n-h)}\int_{B_r}\int_{B_r} {\left| {d_{C}(x,y)}\right|^{-Q-sp}\left| {v(x)-v(y)} \right|^p}\mathrm{d}x\mathrm{d}y,
        \end{aligned}
    \end{equation*}
    where $c\equiv c(n,h,p,s,m)$, which completes the proof.
\end{proof}

We conclude this section by presenting a significant consequence of the Logarithmic Lemma, which will be highly instrumental in Section \ref{section_Proof of the Hölder continuity}.
\begin{corollary}\label{corollary_corollary of logarithmic lemma}
    Let $p\in (1,+\infty)$, $s\in(0,1)$ and let $u\in W^{s,p}\left( {{\mathbb{G}}^{n}} \right)$ be the solution to   \eqref{integro-differential problem} such that $u\ge0$ in $B_R\equiv B\left( {x_0,R} \right)\subset\Omega$. Let $a,d>0,$ $b>1$ and define
    \begin{equation*}
        v:=\min\left\{ {\left( {\log(a+d)-\log(u+d)} \right)_+,\mathrm{log}(b)} \right\}.
    \end{equation*}
    Then the following estimate
    \begin{equation*}
        \begin{aligned}
            &\dashint_{B_r} {\left| {v-(v)_{B_{r}}} \right|^p}\mathrm{d}x   \\
            &\le c\left( {\left| {x_0} \right|_h+r} \right)^{-m(n-h)}\left\{ {d^{1-p}\left( {\frac{r}{R}}
                \right)^{sp}r^{Q-n}\left[ {\mathrm{Tail}\left( {u_{-};x_{0},R} \right)}
                \right]^{p-1}+\left( {\left| {x_0} \right|_h+2r} \right)^{m(n-h)}}
            \right\}
        \end{aligned}
    \end{equation*}holds for any $B_r\equiv B\left( {x_0,r} \right)\subset B\left( {x_0,R/2} \right)$,
    where $c$ depends only on $n,h,p,s$ and $m$.
\end{corollary}
\begin{proof}
    This corollary follows directly from Theorem \ref{lemma_Logarithmic Lemma}. Via the definition of the function $v$, we have
    \begin{equation*}
        \int_{B_r}\int_{B_r} {\left| {d_{C}(x,y)}\right|^{-Q-sp}\left| {v(x)-v(y)} \right|^p}\mathrm{d}x\mathrm{d}y\le\int_{B_r}
        \int_{B_r} {\left| {d_{C}(x,y)}\right|^{-Q-sp}\left| {\log\left( {\frac{u(y)+d}{u(x)+d}} \right)}
            \right|^p}\mathrm{d}x\mathrm{d}y.
    \end{equation*}
    By the fractional Poincar\'{e}-type inequality \eqref{Poincare inequality} in Lemma \ref*{lemma_Poincare inequality}, we apply the estimate in \eqref{Logarithmic Lemma} to deuce
    \begin{equation*}
        \begin{aligned}
            &\dashint_{B_r} {\left| {v-(v)_{B_{r}}} \right|^p}\mathrm{d}x   \\
            \le&cr^{-2n+Q+sp}\left( {\left| {x_0} \right|_h+r} \right)^{-2m(n-h)}\int_{B_{r}}\int_{B_{r}} {\left| {d_{C}(x,y)}\right|^{-Q-sp}\left| {\log\left( {\frac{u(y)+d}{u(x)+d}} \right)} \right|^p}\mathrm{d}x\mathrm{d}y   \\
            \le&c\left( {\left| {x_0} \right|_h+r} \right)^{-m(n-h)}{\left( {\frac{\left| {x_{0}} \right|_h+2r}{\left| {x_{0}} \right|_h+r}} \right)}^{m(n-h)}   \\
            &\times\left\{ {d^{1-p}{\left( {\frac{r}{R}} \right)}^{sp}r^{Q-n}\left[ {\mathrm{Tail}\left( {u_{-};x_{0},R} \right)} \right]^{p-1}+{\left( {\left| {x_{0}} \right|_h+2r} \right)}^{m(n-h)}} \right\}   \\
            \le&c\left( {\left| {x_0} \right|_h+r} \right)^{-m(n-h)}\left\{ {d^{1-p}{\left( {\frac{r}{R}} \right)}^{sp}r^{Q-n}\left[ {\mathrm{Tail}
                    \left( {u_{-};x_{0},R} \right)} \right]^{p-1}+{\left( {\left| {x_{0}} \right|_h+2r} \right)}^{m(n-h)}}
            \right\},
        \end{aligned}
    \end{equation*}
    where we have used the fact that $\frac{\left| {x_{0}} \right|_h+2r}{\left| {x_{0}} \right|_h+r}\le2$. This completes the proof.
\end{proof}
\end{section}

\begin{section}{Proof of the Local Boundedness}\label{section_Proof of the Local Boundedness}
The objective of this section is to establish the local boundedness result for the fractional $p$-minimizers of the functional \eqref{F function in sec2} in Theorem \ref{theo_local boundedness}. Combining precise estimates involving the nonlocal tail of the solutions with the Caccioppoli inequality demonstrated in Section \ref{section_Fundamental estimates}, we establish the local boundedness results on the Grushin-type space; see \cite{di2016local} for the corresponding problem in the Euclidean case and \cite{manfredini2023holder} for the Heisenberg group case.

\textbf{Proof of Theorem \ref{theo_local boundedness}}.
Firstly, it is necessary to define several key quantities. For any $j\in\mathbb{N}$ and $r>0$ such that
$B\left( {x_{0},r} \right)\subset\Omega$, denote by
\begin{equation*}
    \begin{gathered}
        r_j=\frac{1}{2}\left( {1+2^{-j}}
        \right)r,\quad\tilde{r}_j=\frac{r_j+r_{j+1}}{2},\quad
        B_j=B\left( {x_0,r_j} \right),\quad\tilde{B}_j=B\left( {x_0,\tilde{r}_j} \right).
    \end{gathered}
\end{equation*}
Moreover, take
\begin{equation*}
    \begin{gathered}
        \phi_j\in C_{0}^{\infty}\left( {\tilde{B}_j} \right) ,\quad0\le\phi_j\le1,\quad\phi_j\equiv1\,\,
        \mathrm{on}\,\,B_{j+1},\,\,\mathrm{and}\,\,\left| {\nabla_{\mathbb{G}^{n}}\phi_j} \right|<2^{j+3}/r;  \\
        k_j=k+\left( {1-2^{-j}} \right)\tilde{k},\quad\tilde{k}_j=\frac{k_{j+1}+k_j}{2},\quad\tilde{k}\in\mathbb{R}^{+}\,
        \,\mathrm{and}\,\,k\in\mathbb{R}; \\
        \tilde{w}_j=\left( {u-\tilde{k}_j} \right)_{+}\,\,\mathrm{and}\,\,w_j=\left( {u-k_j} \right)_{+}.
    \end{gathered}
\end{equation*}

Secondly, we divide the proof into two steps. We first focus on the subcritical case where $sp<2n-Q$.  We
apply the H\"{o}lder inequality and
the fractional Poincar\'{e} inequality to the function $\tilde{w}_j\phi_j$ as defined above to obtain
\begin{equation*}
    \begin{aligned}
        &\left( {\dashint_{B_j} {\left| {\tilde{w}_j(x)\phi_j(x)} \right|^{p^{*}}}\mathrm{d}x} \right)^{\frac{p}{p^{*}}}  \\
        \le&c\left[ {\left( {\dashint_{B_j} {\left| {\tilde{w}_j\phi_j-\left( {\tilde{w}_j\phi_j} \right)_{B_j}} \right|^{p^{*}}}\mathrm{d}x} \right)^{\frac{1}{p^{*}}}+\dashint_{B_j} {\tilde{w}_j(x)\phi_j(x)}\mathrm{d}x} \right]^p   \\
        \le&c\left [ {\left( {\dashint_{B_j} {\left| {\tilde{w}_j\phi_j-\left( {\tilde{w}_j\phi_j} \right)_{B_j}} \right|^{p^{*}}}\mathrm{d}x} \right)^{\frac{p}{p^{*}}}+\left( {\dashint_{B_j} {\tilde{w}_j(x)\phi_j(x)}\mathrm{d}x} \right)^p} \right]   \\
        \le&c\left[ {\frac{1}{\left| {B_j} \right|}\left( {\int_{B_j} {\left| {\tilde{w}_j\phi_j-\left( {\tilde{w}_j\phi_j} \right)_{B_j}} \right|^{p^{*}\cdot\frac{p}{p^*}}}\mathrm{d}x} \right)^{\frac{p^*}{p}}\left| {B_j} \right|^{1-\frac{p^*}{p}}} \right]^\frac{p}{p^*}+c\left( {\dashint_{B_j} {\tilde{w}_j(x)\phi_j(x)}\mathrm{d}x} \right)^p   \\
        =&c\dashint_{B_j} {\left| {\tilde{w}_j\phi_j-\left( {\tilde{w}_j\phi_j} \right)_{B_j}} \right|^{p}}\mathrm{d}x+c\left( {\dashint_{B_j} {\tilde{w}_j(x)\phi_j(x)}\mathrm{d}x} \right)^p   \\
        \le&cr_j^{-2n+Q+sp}\left( {\left| {x_0} \right|_h+r_j} \right)^{-2m(n-h)}\int_{B_{j}}\int_{B_{j}} {\left| {d_{C}(x,y)}\right|^{-Q-sp}\left| {\tilde{w}_j(x)\phi_j(x)-\tilde{w}_j(y)\phi_j(y)} \right|^p}\mathrm{d}x\mathrm{d}y   \\
        &+c\left( {\dashint_{B_j} {\tilde{w}_j(x)\phi_j(x)}\mathrm{d}x} \right)^p,
    \end{aligned}
\end{equation*}
where $p^{*}=(2n-Q)p/(2n-Q-sp)<p$. Utilizing the nonlocal
Caccioppoli inequality with the tail given by \eqref{Caccioppoli
estimates with tail}, with $w_{+}=\tilde{w}_j$ and $\phi=\phi_j$, we
arrive at \begin{equation}\label{I1+I2+I3 in proof of local
boundedness}
    \begin{aligned}
        &\left( {\dashint_{B_j} {\left| {\tilde{w}_j(x)\phi_j(x)} \right|^{p^{*}}}\mathrm{d}x} \right)^{\frac{p}{p^{*}}}   \\
        \le&cr_j^{-n+Q+sp}\left( {\left| {x_0} \right|_h+r_j} \right)^{-m(n-h)}   \\
        &\times\dashint_{B_{j}}\int_{B_{j}} {\left| {d_{C}(x,y)} \right|^{-Q-sp}{\left( {\max\left\{ {\tilde{w}_{j}(x),\tilde{w}_{j}(y)} \right\}} \right)}^{p}\left| {\phi_j(x)-\phi_j(y)} \right|^{p}}\mathrm{d}x\mathrm{d}y   \\
        &+cr_j^{-n+Q+sp}\left( {\left| {x_0} \right|_h+r_j} \right)^{-m(n-h)}   \\
        &\times\dashint_{B_{r}} {\tilde{w}_{j}(y)\phi^{p}_j(y)\left( {\mathop{\sup}_{y\in \mathrm{{supp}\,\phi_j}}\int_{{\mathbb{G}}^{n}\backslash B_{j}} {\left| {d_{C}(x,y)} \right|^{-Q-sp}\tilde{w}_{j}^{p-1}(x)}\mathrm{d}x} \right)}\mathrm{d}y   \\
        &+c\dashint_{B_j} {\left| {\tilde{w}_{j}(x)\phi_j(x)} \right|^p}\mathrm{d}x   \\
        =:&I_1+I_2+I_3.
    \end{aligned}
\end{equation}

We begin by estimating the first term $I_1$ on the right-hand side of \eqref{I1+I2+I3 in proof of local boundedness}.
Noting that $B_j\subseteq B(y,2r_j)$ for $y\in B_j$ and using the
same strategy used in the proof of the Theorem 1.4, we
obtain
\begin{equation}\label{I1 estimate in proof of local boundedness}
    \begin{aligned}
        I_1\le&c2^{jp}r_j^{-n+Q+sp-p}\left( {\left| {x_0} \right|_h+r_j} \right)^{-m(n-h)}\dashint_{B_{j}} {w_{j}^p(y)\left( {\int_{B_{j}} {\left| {d_{C}(x,y)} \right|^{-Q-sp+p}}\mathrm{d}x} \right)}\mathrm{d}y   \\
        \le&c2^{jp}\dashint_{B_{j}} {w_{j}^p(x)}\mathrm{d}x.
    \end{aligned}
\end{equation}

As for the second term $I_2$ on the right in \eqref{I1+I2+I3 in proof of local boundedness}, we note that
$\tilde{w}_j\le w_j^p/\left( {\tilde{k}_j-k_j} \right)^{p-1}$. Moreover, since $y\in\mathrm{supp}\,\phi_j\subseteq\tilde{B}_j$ and $x\in\mathbb{G}^n\backslash B_j$, we have
\begin{equation*}
    \frac{d_C\left( {x,x_0} \right)}{d_C\left( {x,y} \right)}\le\frac{d_C\left( {x_0,y} \right)+d_C\left( {x,y} \right)}{d_C\left( {x,y} \right)}=1+\frac{d_C\left( {x_0,y} \right)}{d_C\left( {x,y} \right)}\le1+\frac{\tilde{r}_j}{r_j-\tilde{r}_j}\le c2^j.
\end{equation*}
Thus,
\begin{equation}\label{I2 estimate in proof of local boundedness}
    \begin{aligned}
        I_2\le&c2^{j(Q+sp)}r_j^{-n+Q+sp}\left( {\left| {x_0} \right|_h+r} \right)^{-m(n-h)}   \\
        &\times\left( {\dashint_{B_{r}} {\frac{w_{j}^p(y)}{\left( {\tilde{k}_j-k_j} \right)^{p-1}}}\mathrm{d}y} \right)\left( {\int_{{\mathbb{G}}^{n}\backslash B_{j}} {\frac{w_{j}^{p-1}(x)}{\left| {d_{C}(x,x_0)} \right|^{Q+sp}}}\mathrm{d}x} \right)   \\
        \le&c\frac{2^{j(Q+sp+p-1)}}{\tilde{k}^{p-1}r^{n-Q}\left( {\left| {x_0} \right|_h+r} \right)^{m(n-h)}}\left[ {\mathrm{Tail}\left( {w_{0};x_{0},r/2} \right)} \right]^{p-1}\dashint_{B_{j}} {w_{j}^p(y)}\mathrm{d}y.
    \end{aligned}
\end{equation}

Since $\tilde{w}_j\le w_j$ and $\phi\le1$, the third integral $I_3$ in \eqref{I1+I2+I3 in proof of local boundedness} can be estimated as    \begin{equation}\label{I3 estimate in proof of local boundedness}
    I_3\le c\dashint_{B_{j}} {w_{j}^p(x)}\mathrm{d}x.
\end{equation}
Regarding the left-hand side of \eqref{I1+I2+I3 in proof of local boundedness}, we obtain
\begin{equation}\label{left of I1+I2+I3 estimate in proof of local boundedness}
    \begin{aligned}
        &\left( {\dashint_{B_j} {\left| {\tilde{w}_j(x)\phi_j(x)} \right|^{p^{*}}}\mathrm{d}x} \right)^{\frac{p}{p^{*}}}   \\
        \ge&\left( {\frac{\left| {B_{j+1}} \right|}{\left| {B_j} \right|}\dashint_{B_{j+1}} {\left| {\tilde{w}_j(x)\phi_j(x)} \right|^{p^{*}}}\mathrm{d}x} \right)^{\frac{p}{p^{*}}}   \\
        \ge&c\left( {\frac{r_{j+1}^n\left( {\left| {x_0} \right|_h+r_{j+1}} \right)^{m(n-h)}}{r_{j}^n\left( {\left| {x_0} \right|_h+r_{j}} \right)^{m(n-h)}}\dashint_{B_{j+1}} {\left| {\tilde{w}_j(x)\phi_j(x)} \right|^{p^{*}}}\mathrm{d}x} \right)^{\frac{p}{p^{*}}}   \\
        \ge&c\left( {\frac{r_{j+1}}{r_j}} \right)^{n-sp}\left( {\frac{\left| {x_0} \right|_h+r_{j+1}}{\left| {x_0} \right|_h+r_{j}}} \right)^{\frac{m(n-h)p}{p^{*}}}\left( {k_{j+1}-\tilde{k}_j} \right)^{\frac{\left( {p^{*}-p} \right)p}{p^{*}}}\left( {\dashint_{B_{j+1}} {w^{p}_{j+1}(x)}\mathrm{d}x} \right)^{\frac{p}{p^{*}}}   \\
        \ge&c\left( {\frac{\tilde{k}}{2^{j+2}}} \right)^{\frac{\left( {p^{*}-p} \right)p}{p^{*}}}\left( {\dashint_{B_{j+1}} {w^{p}_{j+1}(x)}\mathrm{d}x} \right)^{\frac{p}{p^{*}}},
    \end{aligned}
\end{equation}
where we have used the fact that  $\tilde{w}_j^{p^{*}}\ge\left( {k_{j+1}-\tilde{k}_j} \right)^{p^{*}-p}w_{j+1}^{p}$. Combining
\eqref{I1+I2+I3 in proof of local boundedness} with \eqref{I1 estimate in proof of local boundedness},
\eqref{I2 estimate in proof of local boundedness}, \eqref{I3 estimate in proof of local boundedness} and
\eqref{left of I1+I2+I3 estimate in proof of local
    boundedness} deduces
\begin{equation*}
    \left( {\frac{\tilde{k}^{1-p/p^{*}}}{2^{(j+2)\frac{\left( {p^{*}-p} \right)}{p^{*}}}}} \right)^pA_{j+1}^{\frac{p^2}{p^{*}}}\le c2^{j(Q+sp+p-1)}\left( {2+\frac{\left[ {\mathrm{Tail}\left( {w_{0};x_{0},r/2} \right)} \right]^{p-1}}{\tilde{k}^{p-1}r^{n-Q}\left( {\left| {x_0} \right|_h+r} \right)^{m(n-h)}}} \right)A_j^p,
\end{equation*}
where   $A_j:=\left( {\dashint_{B_{j}} {w_{j}^p(x)}\mathrm{d}x} \right)^{\frac{1}{p}}$.

By choosing
\begin{equation}\label{k require1 in proof of local boundedness}
    \tilde{k}\ge\delta\frac{\mathrm{Tail}\left( {w_{0};x_{0},r/2} \right)}{r^{\frac{n-Q}{p-1}}\left( {\left| {x_0} \right|_h+r} \right)^{\frac{m(n-h)}{p-1}}},\quad\delta\in(0,1],
\end{equation}
we obtain
\begin{equation}\label{iteration1 eq1 in proof of local boundedness}
    \left( {\frac{A_{j+1}}{\tilde{k}}} \right)^{\frac{p}{p^{*}}}\le\delta^{\frac{1-p}{p}}\bar{c}^{\frac{p}{p^{*}}}
    2^{j\left( {\frac{Q+sp+p-1}{p}+\frac{sp}{Q}}
        \right)}\frac{A_{j}}{\tilde{k}},
\end{equation}
where $\bar{c}=2^{\frac{2\left( {p^{*}-p} \right)}{p}}3^{\frac{p^{*}}{p^2}}c^{\frac{p^{*}}{p^2}}$.
Let
\begin{equation*}
    \beta:=\frac{sp}{2n-Q-sp}=\frac{p^{*}}{p}-1,\quad C:=2^{\frac{(2n-Q+sp+p-1)(2n-Q)}{p(2n-Q-sp)}+\frac{sp}{2n-Q-sp}}>1,
\end{equation*}
allowing the estimate in \eqref{iteration1 eq1 in proof of local boundedness} to be rewritten as
\begin{equation}\label{iteration1 eq2 in proof of local boundedness}
    \frac{A_{j+1}}{\tilde{k}}\le\delta^{\frac{p^{*}(p-1)}{p^2}}\bar{c}C^j\left( {\frac{A_{j}}{\tilde{k}}} \right)^{1+\beta}.
\end{equation}
Therefore, it is sufficient to require
\begin{equation*}
    \frac{A_0}{\tilde{k}}\le\delta^{\frac{p^{*}(p-1)}{p^2\beta}}\bar{c}^{-\frac{1}{\beta}}C^{-\frac{1}{\beta^2}},
\end{equation*}
and thus we obtain
\begin{equation*}
    \frac{A_j}{\tilde{k}}\le\delta^{\frac{p^{*}(p-1)}{p^2\beta}}\bar{c}^{-\frac{1}{\beta}}C^{-\frac{j\beta+1}{\beta^2}}
\end{equation*}
through the iteration relation \eqref{iteration1 eq2 in proof of local boundedness}, which yields $A_j\to0$ as $j\to\infty$. Since
\begin{equation*}
    \frac{p^{*}(p-1)}{p^2\beta}=\frac{(2n-Q)(p-1)}{sp^2},
\end{equation*}
we set
\begin{equation*}
    \tilde{k}=\delta \frac{\mathrm{Tail} \left( {w_{0};x_{0},r/2} \right)}{r^{\frac{n-Q}{p-1}}\left( {{\left| {x_{0}} \right|_h}+r} \right)^{\frac{m(n-h)}{p-1}}}+c\delta ^{-\frac{(2n-Q)(p-1)}{sp^2}}HA_0,\quad H=\bar{c}^{\frac{1}{\beta}}C^{\frac{1}{\beta^2}},
\end{equation*}
which coincides with \eqref{k require1 in proof of local boundedness}. It follows that
\begin{equation*}
    \begin{aligned}
        \mathop{\sup}_{B\left( {x_{0},r/2} \right)} u\le&\lim_{j\to\infty}k_j   \\
        &=k+\tilde{k}   \\
        &\le k+\delta \frac{\mathrm{Tail} \left( {(u-k)_+;x_{0},r/2} \right)}{r^{\frac{n-Q}{p-1}}\left( {{\left| {x_{0}} \right|_h}+r} \right)^{\frac{m(n-h)}{p-1}}}+\delta^{-\frac{(2n-Q)(p-1)}{sp^2}}H{\left( {\dashint_{B_r} {(u-k)_{+}^{p}}\mathrm{d}x} \right)}^{\frac{1}{p}},
    \end{aligned}
\end{equation*}
which yields the desired result by choosing $k=0$. Thus, the proof is complete for the case $sp<2n-Q$.

We now turn to the remaining case:  $sp=2n-Q$. We use the argument of
the specific procedure  provided in \cite{manfredini2023holder}  to complete the proof. Choose $0<s_1<s<1$, in particular, $\tilde{w}_j\phi_j-\left( {\tilde{w}_j\phi_j} \right)_{B_j}\in W^{s,p}\left( {B_j} \right)\subseteq W^{s_1,p}\left( {B_j} \right)$. Clearly, $s_1p<sp=2n-Q$, and we are justified in applying the H{\"o}lder inequality, which yields, for any $p<q<p^{*}_1=\frac{(2n-Q)p}{2n-Q-s_1p}$,
\begin{equation*}
    \left\| {\tilde{w}_j\phi_j-\left( {\tilde{w}_j\phi_j} \right)_{B_j}} \right\|_{L^q\left( {B_j} \right)}\le\left| {B_j} \right|^{\frac{p_1^{*}-q}{qp^{*}_1}}\left\| {\tilde{w}_j\phi_j-\left( {\tilde{w}_j\phi_j} \right)_{B_j}} \right\|_{L^{p_1^{*}}\left( {B_j} \right)}.
\end{equation*}
Thus, by the fractional Poincar{\'e} type inequality in Lemma \ref{lemma_Poincare inequality} we can express the result as
\begin{equation}\label{1st step of sp=n in proof of local boundedness}
    \begin{aligned}
        &\left( {\dashint_{B_j} {\left| {\tilde{w}_j(x)\phi_j(x)} \right|^{q}}\mathrm{d}x} \right)^{\frac{p}{q}}  \\
        \le&c\left[ {\left( {\dashint_{B_j} {\left| {\tilde{w}_j\phi_j-\left( {\tilde{w}_j\phi_j} \right)_{B_j}} \right|^{q}}\mathrm{d}x} \right)^{\frac{1}{q}}+\dashint_{B_j} {\tilde{w}_j(x)\phi_j(x)}\mathrm{d}x} \right]^p   \\
        \le&c\left [ {\left( {\dashint_{B_j} {\left| {\tilde{w}_j\phi_j-\left( {\tilde{w}_j\phi_j} \right)_{B_j}} \right|^{q}}\mathrm{d}x} \right)^{\frac{p}{q}}+\left( {\dashint_{B_j} {\tilde{w}_j(x)\phi_j(x)}\mathrm{d}x} \right)^p} \right]   \\
        \le&c\left[ {\left| {B_j} \right|^{\frac{\left( {p_1^{*}-q} \right)p}{qp_1^{*}}-\frac{p}{q}}\left( {\int_{B_j} {\left| {\tilde{w}_j\phi_j-\left( {\tilde{w}_j\phi_j} \right)_{B_j}} \right|^{p_1^{*}}}\mathrm{d}x} \right)^{\frac{p}{p_1^{*}}}+\left( {\dashint_{B_j} {\tilde{w}_j(x)\phi_j(x)}\mathrm{d}x} \right)^p} \right]   \\
        \le&c\left[ {\frac{1}{\left| {B_j} \right|}\left( {\int_{B_j} {\left| {\tilde{w}_j\phi_j-\left( {\tilde{w}_j\phi_j} \right)_{B_j}} \right|^{p_1^{*}\cdot\frac{p}{p_1^*}}}\mathrm{d}x} \right)^{\frac{p_1^*}{p}}\left| {B_j} \right|^{1-\frac{p_1^*}{p}}} \right]^\frac{p}{p_1^*}+c\left( {\dashint_{B_j} {\tilde{w}_j(x)\phi_j(x)}\mathrm{d}x} \right)^p   \\
        =&c\dashint_{B_j} {\left| {\tilde{w}_j\phi_j-\left( {\tilde{w}_j\phi_j} \right)_{B_j}} \right|^{p}}\mathrm{d}x+c\left( {\dashint_{B_j} {\tilde{w}_j(x)\phi_j(x)}\mathrm{d}x} \right)^p   \\
        \le&c\left( {\dashint_{B_j} {\tilde{w}_j(x)\phi_j(x)}\mathrm{d}x} \right)^p+cr_j^{-2n+Q+s_1p}\left( {\left| {x_0} \right|_h+r_j} \right)^{-2m(n-h)}   \\
        &\times\int_{B_j}\int_{B_j} {\left| {d_{C}(x,y)}\right|^{-n-s_1p}\left| {\tilde{w}_j(x)\phi_j(x)-\tilde{w}_j(y)\phi_j(y)} \right|^p}\mathrm{d}x\mathrm{d}y.   \\
    \end{aligned}
\end{equation}

We are now in a position to apply the nonlocal Caccioppoli-type inequality in \eqref{1st step of sp=n in proof of local boundedness}, as previously done in the subcritical case. Similarly, by combining the above inequality, we obtain
\begin{align}\label{J1+J2+J3 in proof of local boundedness}
        &\left( {\dashint_{B_j} {\left| {\tilde{w}_j(x)\phi_j(x)} \right|^{q}}\mathrm{d}x} \right)^{\frac{p}{q}}   \notag\\
        \le&cr_j^{-n+Q+s_1p}\left( {\left| {x_0} \right|_h+r_j} \right)^{-m(n-h)}   \notag\\
        &\times\dashint_{B_{j}}\int_{B_{j}} {\left| {d_{C}(x,y)} \right|^{-Q-s_1p}{\left( {\max\left\{ {\tilde{w}_{j}(x),\tilde{w}_{j}(y)} \right\}} \right)}^{p}\left| {\phi_j(x)-\phi_j(y)} \right|^{p}}\mathrm{d}x\mathrm{d}y   \notag\\
        &+cr_j^{-n+Q+s_1p}\left( {\left| {x_0} \right|_h+r_j} \right)^{-m(n-h)}   \notag\\
        &\times\dashint_{B_{r}} {\tilde{w}_{j}(y)\phi^{p}_j(y)\left( {\mathop{\sup}_{y\in \mathrm{{supp}\,\phi_j}}\int_{{\mathbb{G}}^{n}\backslash B_{j}} {\left| {d_{C}(x,y)} \right|^{-Q-s_1p}\tilde{w}_{j}^{p-1}(x)}\mathrm{d}x} \right)}\mathrm{d}y   \notag\\
        &+c\dashint_{B_j} {\left| {\tilde{w}_{j}(x)\phi_j(x)} \right|^p}\mathrm{d}x   \notag\\
        :=&J_1+J_2+J_3.
\end{align}
The estimates for $J_1$, $J_2$ and $J_3$ in \eqref{J1+J2+J3 in proof of local boundedness} are analogous to those previously discussed in the subcritical case. As for $J_1$, by the definition of $\phi_j$, we have
\begin{equation}\label{J1 estimate in proof of local boundedness}
    \begin{aligned}
        J_1\le&c2^{jp}r_j^{-n+Q+s_1p-p}\left( {\left| {x_0} \right|_h+r_j} \right)^{-m(n-h)}\dashint_{B_{j}} {w_{j}^p(y)\left( {\int_{B_{j}} {\left| {d_{C}(x,y)} \right|^{-Q-s_1p+p}}\mathrm{d}x} \right)}\mathrm{d}y   \\
        \le&c2^{jp}r_j^{-n+Q+s_1p-p}\left( {\left| {x_0} \right|_h+r_j} \right)^{-m(n-h)}r_j^{n-Q+p-s_1p}\left( {\left| {x_0} \right|_h+r_j} \right)^{m(n-h)}\dashint_{B_{j}} {w_{j}^p(x)}\mathrm{d}x   \\
        \le&c2^{jp}\dashint_{B_{j}} {w_{j}^p(x)}\mathrm{d}x.
    \end{aligned}
\end{equation}
For $J_2$, we obtain
\begin{equation}\label{J2 estimate in proof of local boundedness}
    \begin{aligned}
        J_2\le&c2^{j(Q+s_1p)}r_j^{-n+Q+s_1p}\left( {\left| {x_0} \right|_h+r} \right)^{-m(n-h)}   \\
        &\times\left( {\dashint_{B_{r}} {\frac{w_{j}^p(y)}{\left( {\tilde{k}_j-k_j} \right)^{p-1}}}\mathrm{d}y} \right)\left( {\int_{{\mathbb{G}}^{n}\backslash B_{j}} {\frac{w_{j}^{p-1}(x)}{\left| {d_{C}(x,x_0)} \right|^{Q+s_1p}}}\mathrm{d}x} \right)   \\
        \le&c\frac{2^{j(Q+s_1p+p-1)}}{\tilde{k}^{p-1}r^{n-Q}\left( {\left| {x_0} \right|_h+r} \right)^{m(n-h)}}\left[ {\mathrm{Tail}\left( {w_{0};x_{0},r/2} \right)} \right]^{p-1}\dashint_{B_{j}} {w_{j}^p(y)}\mathrm{d}y.
    \end{aligned}
\end{equation}
And for $J_3$, it can be easily estimated as
\begin{equation}\label{J3 estimate in proof of local boundedness}
    J_3\le c\dashint_{B_{j}} {w_{j}^p(x)}\mathrm{d}x
\end{equation}
since $\tilde{w}_j\le w_j$ and $\phi\le1$. By combining \eqref{J1+J2+J3 in proof of local boundedness} with \eqref{J1 estimate in proof of local boundedness}, \eqref{J2 estimate in proof of local boundedness} and \eqref{J3 estimate in proof of local boundedness}, we obtain
\begin{equation}\label{J1+J2+J3 estimate in proof of local boundedness}
    \left( {\dashint_{B_j} {\left| {\tilde{w}_j(x)\phi_j(x)} \right|^{q}}\mathrm{d}x} \right)^{\frac{p}{q}}\le c2^{j(Q+s_1p+p-1)}\left( {2+\frac{\left[ {\mathrm{Tail}\left( {w_{0};x_{0},r/2} \right)} \right]^{p-1}}{\tilde{k}^{p-1}r^{n-Q}\left( {\left| {x_0} \right|_h+r} \right)^{m(n-h)}}} \right)\dashint_{B_{j}} {w_{j}^p(x)}\mathrm{d}x.
\end{equation}

Moreover, the term on the left-hand side of the inequality above can be estimated analogously to the subcritical case in \eqref{left of I1+I2+I3 estimate in proof of local boundedness}, as follows:
\begin{equation}\label{left of J1+J2+J3 estimate in proof of local boundedness}
    \left( {\dashint_{B_j} {\left| {\tilde{w}_j(x)\phi_j(x)} \right|^{q}}\mathrm{d}x} \right)^{\frac{p}{q}}\ge c\left( {\frac{\tilde{k}}{2^{j+2}}} \right)^{\frac{\left( {q-p} \right)p}{q}}\left( {\dashint_{B_{j+1}} {w^{p}_{j+1}(x)}\mathrm{d}x} \right)^{\frac{p}{q}}.
\end{equation}
Similarly, we set $A_j:=\left( {\dashint_{B_{j}} {w_{j}^p(x)}\mathrm{d}x} \right)^{\frac{1}{p}}$ and
choose $\tilde{k}$ as in \eqref{k require1 in proof of local boundedness}.
By combining \eqref{J1+J2+J3 estimate in proof of local boundedness} and
\eqref{left of J1+J2+J3 estimate in proof of local boundedness}, we
obtain
\begin{equation}\label{iteration2 eq2 in proof of local boundedness}
    \frac{A_{j+1}}{\tilde{k}}\le\delta^{\frac{q(p-1)}{p^2}}\bar{c}C^j\left( {\frac{A_{j}}{\tilde{k}}}
    \right)^{1+\beta},
\end{equation}
where
\begin{equation*}
    \bar{c}=2^{\frac{2\left( {q-p} \right)}{p}}3^{\frac{q}{p^2}}c^{\frac{q}{p^2}},
    C:=2^{\frac{(Q+s_1p+p-1)q}{p^2}+\frac{q-p}{p}}>1,
    \beta:=\frac{q}{p}-1.
\end{equation*}
As previously done in the subcritical case, it is sufficient to require
\begin{equation*}
    \frac{A_0}{\tilde{k}}\le\delta^{\frac{q(p-1)}{p^2\beta}}\bar{c}^{-\frac{1}{\beta}}C^{-\frac{1}{\beta^2}}
\end{equation*}
to obtain
\begin{equation*}
    \frac{A_j}{\tilde{k}}\le\delta^{\frac{q(p-1)}{p^2\beta}}\bar{c}^{-\frac{1}{\beta}}C^{-\frac{j\beta+1}{\beta^2}}
\end{equation*}
through the iteration relation \eqref{iteration2 eq2 in proof of local boundedness}, which yields $A_j\to0$ as $j\to\infty$. Since
\begin{equation*}
    \frac{q(p-1)}{p^2\beta}=\frac{q(p-1)}{p(q-p)},
\end{equation*}
we set
\begin{equation*}
    \tilde{k}=\delta \frac{\mathrm{Tail} \left( {w_{0};x_{0},r/2} \right)}{r^{\frac{n-Q}{p-1}}\left( {{\left| {x_{0}} \right|_h}+r} \right)^{\frac{m(n-h)}{p-1}}}+c\delta^{-\frac{q(p-1)}{p(q-p)}}HA_0,\quad H=\bar{c}^{-\frac{1}{\beta}}C^{\frac{1}{\beta^2}},
\end{equation*}
which is in accordance with \eqref{k require1 in proof of local boundedness}. We finally deduce that
\begin{equation*}
    \mathop{\sup}_{B\left( {x_{0},r/2} \right)} u\le k+\delta \frac{\mathrm{Tail} \left( {(u-k)_+;x_{0},r/2} \right)}{r^{\frac{n-Q}{p-1}}\left( {{\left| {x_{0}} \right|_h}+r} \right)^{\frac{m(n-h)}{p-1}}}+\delta^{-\frac{q(p-1)}{p(q-p)}}H{\left( {\dashint_{B_r} {(u-k)_{+}^{p}}\mathrm{d}x} \right)}^{\frac{1}{p}},
\end{equation*}
which gives the desired result by choosing $k=0$.

\end{section}

\begin{section}{Proof of the H{\"o}lder continuity}\label{section_Proof of the Hölder continuity}
This final section is devoted  to proving the H{\"o}lder continuity of solutions, specifically Theorem \ref{theo_holder continuity}. It is worth noting that all the estimates established in the previous sections will be utilized.

Before proceeding, we first give some notation. For any $j\in\mathbb{N}$, let $0<r<R/2$, for some $R$ such that $B\left( {x_0,R} \right)\subset \Omega$, let
\begin{equation*}
    r_j:=\sigma^j\frac{r}{2},\quad\sigma\in(0,1/4],\quad B_j:=B\left( {x_0,r_j} \right).
\end{equation*}
Moreover, we proceed by defining
\begin{equation*}
    \frac{1}{2}\omega\left( {r_0} \right):=\frac{1}{2}\omega\left( {\frac{r}{2}} \right) =\frac{\mathrm{Tail} \left( {u;x_{0},r/2} \right)}{r^{\frac{n-Q}{p-1}}\left( {{\left| {x_{0}} \right|_h}+r} \right)^{\frac{m(n-h)}{p-1}}}+c{\left( {\dashint_{B\left( {x_{0},r} \right)} {\left| {u} \right|^{p}}\mathrm{d}x} \right)}^{\frac{1}{p}},
\end{equation*}
with $\mathrm{Tail} \left( {u;x_{0},r/2} \right)$ as in
\eqref{nonlocal tail} an $c$ as in \eqref{local boundedness}, and
\begin{equation*}
    \omega\left( {r_j} \right):= \left( {\frac{r_j}{r_0}} \right)^{\alpha}\omega\left( {r_0} \right),\quad\mathrm{for}\,\,\mathrm{some}\,\,\alpha <\frac{sp}{p-1}.
\end{equation*}

To establish Theorem \ref{theo_holder continuity}, it suffices to prove the following lemma.
\begin{lemma}
    Using the notation introduced above, let $u\in W^{s,p}\left( {{\mathbb{G}}^{n}} \right)$ be the solution to \eqref{integro-differential problem in sec2}. Then
    \begin{equation}\label{osc of u estimate}
        \mathop{\mathrm{osc}}_{B_j}u\equiv\mathop{\sup}_{B_j}u-\mathop{\inf}_{B_j}u\le\omega\left( {r_j} \right) \quad\forall\,
        j=0,1,2,\dots.
    \end{equation}
\end{lemma}
\begin{proof}
    We proceed by induction. Note that, by the definition of $\omega\left( {r_0} \right)$ and Theorem \ref{theo_local boundedness} (with $\delta=1$ there), the estimate in \eqref{osc of u estimate} holds trivially for $j=0$, as both $(u)_+$ and $(-u)_+$ are weak subsolutions.

    We now make a strong induction assumption, assuming that \eqref{osc of u estimate} holds for all $i\in\{0,\dots,j\}$ for some $j\ge0$. We then demonstrate that it also holds for $j+1$. We have that either
    \begin{equation}\label{u > inf 0.5 in proof of holder continuity}
        \frac{\left| {2B_{j+1}\cap\left\{ {u\ge\inf_{B_j}u+\omega\left( {r_j} \right)/2} \right\}} \right|}{\left| {2B_{j+1}}
            \right|}\ge\frac{1}{2},
    \end{equation} or
    \begin{equation}\label{u < inf 0.5 in proof of holder continuity}
        \frac{\left| {2B_{j+1}\cap\left\{ {u\le\inf_{B_j}u+\omega\left( {r_j} \right)/2} \right\}} \right|}{\left| {2B_{j+1}}
            \right|}\ge\frac{1}{2}
    \end{equation}
    must hold. If \eqref{u > inf 0.5 in proof of holder continuity} holds, we define $u_j:=u-\inf_{B_j}u$; if \eqref{u < inf 0.5 in proof of holder continuity} holds, we define $u_j:=\omega\left( {r_j} \right)-\left( {u-\inf_{B_j}u} \right)$. Consequently, in all cases, we have $u_j\ge0$ in $B_j$, and
    \begin{equation}\label{uj > omega 0.5 in proof of holder continuity}
        \frac{\left| {2B_{j+1}\cap\left\{ {u_j\ge\omega\left( {r_j} \right)/2} \right\}} \right|}{\left| {2B_{j+1}} \right|}\ge\frac{1}{2}.
    \end{equation}
    Moreover, $u_j$ is a weak solution that satisfies:
    \begin{equation}\label{uj < omega in proof of holder continuity}
        \mathop{\sup}_{B_i}\left| {u_j} \right|\le2\omega\left( {r_i} \right) \quad\forall\, i\in\{0,\dots,j\}.
    \end{equation}
    We now claim that, under the induction hypothesis, we obtain
    \begin{equation}\label{tail < omega in proof of holder continuity}
        \left[ {\mathrm{Tail} \left( {u_j;x_{0},r_j} \right)} \right]^{p-1}\le c\sigma^{-\alpha(p-1)}\left[ {\omega\left( {r_j} \right)} \right]^{p-1},
    \end{equation}
    where the constant $c$ depends only on $n,p,s$ and the difference of $sp/(p-1)$ and $\alpha$; however, it is independent of $\sigma$. Indeed, we have
    \begin{equation}\label{tail < omega step1 in proof of holder continuity}
        \begin{aligned}
            \left[ {\mathrm{Tail} \left( {u_j;x_{0},r_j} \right)} \right]^{p-1}=&r_j^{sp}\sum_{i=1}^{j}\int_{B_{i-1}\backslash B_i} {\left| {u_j(x)} \right|^{p-1}\left| {d_C\left( {x,x_0} \right)} \right|^{-Q-sp}}\mathrm{d}x   \\
            &+r_j^{sp}\int_{\mathbb{G}^n\backslash B_0} {\left| {u_j(x)} \right|^{p-1}\left| {d_C\left( {x,x_0} \right)} \right|^{-Q-sp}}\mathrm{d}x   \\
            \le&r_j^{sp}\sum_{i=1}^{j} \left[ {\mathop{\sup}_{B_{i-1}}\left| {u_j} \right|} \right]^{p-1}\int_{\mathbb{G}^n\backslash B_i} {\left| {d_C\left( {x,x_0} \right)} \right|^{-Q-sp}}\mathrm{d}x   \\
            &+r_j^{sp}\int_{\mathbb{G}^n\backslash B_0} {\left| {u_j(x)} \right|^{p-1}\left| {d_C\left( {x,x_0} \right)} \right|^{-Q-sp}}\mathrm{d}x   \\
            \le&c\sum_{i=1}^{j}\left( {\frac{r_j}{r_i}} \right)^{sp}\left[ {\omega\left( {r_{i-1}} \right)}
            \right]^{p-1},
        \end{aligned}
    \end{equation}
    where in the last line we also have used \eqref{uj < omega in proof of holder continuity} and the fact that
    \begin{equation*}
        \begin{aligned}
            &\int_{\mathbb{G}^n\backslash B_0} {\left| {u_j(x)} \right|^{p-1}\left| {d_C\left( {x,x_0} \right)} \right|^{-Q-sp}}\mathrm{d}x   \\
            \le&cr_0^{-sp}\mathop{\sup}_{B_0}\left| {u} \right|^{p-1}+cr_0^{-sp}\left[ {\omega\left( {r_0} \right)} \right]^{p-1}+c\int_{\mathbb{G}^n\backslash B_0} {\left| {u(x)} \right|^{p-1}\left| {d_C\left( {x,x_0} \right)} \right|^{-Q-sp}}\mathrm{d}x   \\
            \le&cr_1^{-sp}\left[ {\omega\left( {r_0} \right)} \right]^{p-1}.
        \end{aligned}
    \end{equation*}
    Moreover, we have
    \begin{equation}\label{tail < omega step2 in proof of holder continuity}
        \begin{aligned}
            &\sum_{i=1}^{j}\left( {\frac{r_j}{r_i}} \right)^{sp}\left[ {\omega\left( {r_{i-1}} \right)} \right]^{p-1}   \\
            =&\left[ {\omega\left( {r_{0}} \right)} \right]^{p-1}\left( {\frac{r_j}{r_0}} \right)^{\alpha(p-1)}\sum_{i=1}^{j}\left( {\frac{r_{i-1}}{r_i}} \right)^{\alpha(p-1)}\left( {\frac{r_j}{r_i}} \right)^{sp-\alpha(p-1)}   \\
            =&\left[ {\omega\left( {r_{j}} \right)} \right]^{p-1}\sigma^{-\alpha(p-1)}\sum_{i=0}^{j-1}\sigma^{i(sp-\alpha(p-1))}   \\
            \le&\left[ {\omega\left( {r_{j}} \right)} \right]^{p-1}\frac{\sigma^{-\alpha(p-1)}}{1-\sigma^{sp-\alpha(p-1)}}   \\
            \le&\frac{4^{sp-\alpha(p-1)}}{\log(4)(sp-\alpha(p-1))}\sigma^{-\alpha(p-1)}\left[ {\omega\left( {r_{j}} \right)}
            \right]^{p-1},
        \end{aligned}
    \end{equation}
    where we used the fact that $\sigma\le1/4$ and $\alpha<sp/(p-1)$. Combining \eqref{tail < omega step1 in proof of holder continuity}
    with \eqref{tail < omega step2 in proof of holder continuity}, we deduces  the desired estimate in \eqref{tail < omega in proof of holder continuity}.

    Next, we consider the function $v$ defined by
    \begin{equation}\label{v def in proof of holder continuity}
        v:=\min\left\{ {\left[ {\log\left( {\frac{\omega\left( {r_j} \right)/2+d}{u_j+d}} \right)} \right]_+,k} \right\},\quad k>0.
    \end{equation}
    By applying Corollary \ref{corollary_corollary of logarithmic lemma},  with $a\equiv\omega\left( {r_j} \right)/2$ and $b\equiv \exp(k)$, we obtain
    \begin{equation*}
        \begin{aligned}
            &\dashint_{2B_{j+1}} {\left| {v-(v)_{2B_{j+1}}} \right|^p}\mathrm{d}x   \\
            &\le c\left( {\left| {x_0} \right|_h+2r_{j+1}} \right)^{-m(n-h)}\left\{ {d^{1-p}\left( {\frac{r_{j+1}}{r_j}} \right)^{sp}r_{j+1}^{Q-n}\left[ {\mathrm{Tail}\left( {u_{j};x_{0},r_j} \right)} \right]^{p-1}+\left( {\left| {x_0} \right|_h+4r_{j+1}} \right)^{m(n-h)}} \right\}.
        \end{aligned}
    \end{equation*}
    Thus, as a consequence of the estimate in \eqref{tail < omega in proof of holder continuity}, we obtain
    \begin{equation*}
        \begin{aligned}
            &\dashint_{2B_{j+1}} {\left| {v-(v)_{2B_{j+1}}} \right|^p}\mathrm{d}x   \\
            &\le c\left( {\left| {x_0} \right|_h+2r_{j+1}} \right)^{-m(n-h)}\left\{ {d^{1-p}\sigma^{sp-\alpha(p-1)}r_{j+1}^{Q-n}\left[ {\omega\left( {r_j} \right)} \right]^{p-1}+\left( {\left| {x_0} \right|_h+4r_{j+1}} \right)^{m(n-h)}} \right\}   \\
            &\le c\left\{ {\left( {\left| {x_0} \right|_h+2r_{j+1}} \right)^{-m(n-h)}d^{1-p}\sigma^{sp-\alpha(p-1)}r_{j+1}^{Q-n}\left[ {\omega\left( {r_j} \right)} \right]^{p-1}+1} \right\}.
        \end{aligned}
    \end{equation*}
    Therefore, by choosing $d=\epsilon\omega\left( {r_j} \right)r_{j+1}^{\frac{Q-n}{p-1}}\left( {\left| {x_0} \right|_h+2r_{j+1}} \right)^{\frac{-m(n-h)}{p-1}}$ with
    \begin{equation}\label{epsilon in proof of holder continuity}
        \epsilon:=\sigma^{\frac{sp}{p-1}-\alpha},
    \end{equation}
    we obtain
    \begin{equation}\label{v < c in proof of holder continuity}
        \dashint_{2B_{j+1}} {\left| {v-(v)_{2B_{j+1}}} \right|^p}\mathrm{d}x\le c,
    \end{equation}
    where the constant $c$ depends only on $n,h,p,s,m$ and the difference of $sp/(p-1)$ and $\alpha$.

    To proceed, let us denote $\tilde{B}\equiv2B_{j+1}$ for brevity. In accordance with \eqref{uj > omega 0.5 in proof of holder continuity} and \eqref{v def in proof of holder continuity}, we can express
    \begin{equation*}
        \begin{aligned}
            k=&\frac{1}{\left| {\tilde{B}\cap\left\{ {u_j\ge\omega\left( {r_j} \right)/2} \right\}} \right|}\int_{\tilde{B}\cap\left\{ {u_j\ge\omega\left( {r_j} \right)/2} \right\}} {k}\mathrm{d}x   \\
            =&\frac{1}{\left| {\tilde{B}\cap\left\{ {u_j\ge\omega\left( {r_j} \right)/2} \right\}} \right|}\int_{\tilde{B}\cap\left\{ {v=0} \right\}} {k}\mathrm{d}x   \\
            \le&\frac{2}{\left| {\tilde{B}} \right|}\int_{\tilde{B}} {(k-v)}\mathrm{d}x=2\left[ {k-(v)_{\tilde{B}}} \right].
        \end{aligned}
    \end{equation*}
    By integrating the preceding inequality over the set $\tilde{B}\cap\left\{ {v=k} \right\}$, we obtain
    \begin{equation*}
        \begin{aligned}
            \frac{\left| {\tilde{B}\cap\left\{ {v=k} \right\}} \right|}{\left| {\tilde{B}} \right|}k\le&\frac{2}{\left| {\tilde{B}} \right|}\int_{\tilde{B}\cap\left\{ {v=k} \right\}} {\left[ {k-(v)_{\tilde{B}}} \right]}\mathrm{d}x   \\
            \le&\frac{2}{\left| {\tilde{B}} \right|}\int_{\tilde{B}} {\left[ {v-(v)_{\tilde{B}}} \right]}\mathrm{d}x\le c,
        \end{aligned}
    \end{equation*}
    where we also have used \eqref{v < c in proof of holder continuity}. Let us take
    \begin{equation*}
        \begin{aligned}
            k&=\log\left( {\frac{\omega\left( {r_j} \right)/2+\epsilon r_{j+1}^{\frac{Q-n}{p-1}}\left( {\left| {x_0} \right|_h+2r_{j+1}} \right)^{\frac{-m(n-h)}{p-1}}\omega\left( {r_j} \right)}{\left[ {2+r_{j+1}^{\frac{Q-n}{p-1}}\left( {\left| {x_0} \right|_h+2r_{j+1}} \right)^{\frac{-m(n-h)}{p-1}}} \right]\epsilon\omega\left( {r_j} \right)}} \right)   \\
            &=\log\left( {\frac{1/2+\epsilon r_{j+1}^{\frac{Q-n}{p-1}}\left( {\left| {x_0} \right|_h+2r_{j+1}} \right)^{\frac{-m(n-h)}{p-1}}}{\left[ {2+r_{j+1}^{\frac{Q-n}{p-1}}\left( {\left| {x_0} \right|_h+2r_{j+1}} \right)^{\frac{-m(n-h)}{p-1}}} \right]\epsilon}} \right)   \\
            &\approx\log\frac{1}{\epsilon},
        \end{aligned}
    \end{equation*}
    so that
    \begin{equation*}
        \frac{\left| {\tilde{B}\cap\left\{ {v=k} \right\}} \right|}{\left| {\tilde{B}} \right|}k\le
        c,
    \end{equation*}
    which yields that
    \begin{equation}\label{c/k < clog in proof of holder continuity}
        \frac{\left| {\tilde{B}\cap\left\{ {u_j\le2\epsilon\omega\left( {r_j} \right)} \right\}} \right|}{\left| {\tilde{B}} \right|}\le \frac{c}{k}\le\frac{c_{\log}}{\log\left( {\frac{1}{\sigma}} \right)},
    \end{equation}
    where the constant $c_{\log}$ depends only on $n,h,p,s,m$ and the difference of $sp/(p-1)$ and $\alpha$ via the definition of $\epsilon$ in \eqref{epsilon in proof of holder continuity}.

    We are now prepared to initiate an appropriate iteration to deduce the desired oscillation reduction. First, for any $i=0,1,2,\dots$, we define
    \begin{equation*}
        \varrho_i=r_{j+1}+2^{-i}r_{j+1},\quad\tilde{\varrho}_i:=\frac{\varrho_i+\varrho_{i+1}}{2},\quad B^i=B_{\varrho_i},\quad\tilde{B}^i=B_{\tilde{\varrho}_i}
    \end{equation*}
    and the corresponding cut-off functions
    \begin{equation*}
        \phi_i\in C_{0}^{\infty}\left( {\tilde{B}^i} \right) ,\quad0\le\phi_i\le1,\quad\phi_i\equiv1\,\,\mathrm{on}\,\,B^{i+1},\,\,\mathrm{and}\,\,\left| {\nabla_{\mathbb{G}^{n}}\phi_i} \right|<c\varrho_i^{-1}.
    \end{equation*}
    Furthermore, set
    \begin{equation*}
        k_i=\left( {1+2^{-i}} \right)\epsilon\omega\left( {r_j} \right),\quad w_i:=\left( {k_i-u_j} \right)_+,
    \end{equation*}
    and
    \begin{equation*}
        A_i=\frac{\left| {B^i\cap\left\{ {u_j\le k_i} \right\}} \right|}{\left| {B^i} \right|}.
    \end{equation*}
    We now focus on the subcritical case where $sp<Q$. The Caccioppoli inequality in \eqref{Caccioppoli estimates with tail} yields
    \begin{equation}\label{use Caccioppoli inequality 1 in proof of holder continuity}
        \begin{aligned}
            &\int_{B^i}\int_{B^i} {\left| {d_{C}(x,y)} \right|^{-Q-sp}\left| {w_{i}(x)\phi_i(x)-w_{i}(y)\phi_i(y)} \right|^{p}}\mathrm{d}x\mathrm{d}y   \\
            \le &c\int_{B^i}\int_{B^i} {\left| {d_{C}(x,y)} \right|^{-Q-sp}{\left( {\mathrm{max}\left\{ {w_{i}(x),w_{i}(y)} \right\}} \right)}^{p}\left| {\phi_i(x)-\phi_i(y)} \right|^{p}}\mathrm{d}x\mathrm{d}y   \\
            &+c\int_{B^i} {w_{i}(x)\phi_i^{p}(x)}\mathrm{d}x\left( {\mathop{\sup}_{y\in\tilde{B}^i}\int_{{\mathbb{G}}^{n}\backslash B^i}
                {\left| {d_{C}(x,y)} \right|^{-Q-sp}w_{i}^{p-1}(x)}\mathrm{d}x}
            \right).
        \end{aligned}
    \end{equation}
    Setting $p^*=\frac{(2n-Q)p}{2n-Q-sp}$, we can estimate the term on the left-hand side as follows:
    \begin{equation}\label{left-hand term estimate 1 in proof of holder continuity}
        \begin{aligned}
            &A_{i+1}^{\frac{p}{p^*}}\left( {k_i-k_{i+1}} \right)^p   \\
            =&\frac{1}{\left| {B^{i+1}} \right|^{\frac{p}{p^*}}}\left( {\int_{B^{i+1}\cap\left\{ {u_j\le k_{i+1}} \right\}} {\left( {k_i-k_{i+1}} \right)^{p^*}\phi_i^{p^{*}}(x)}\mathrm{d}x} \right)^{\frac{p}{p^*}}   \\
            \le&\frac{1}{\left| {B^{i+1}} \right|^{\frac{p}{p^*}}}\left( {\int_{B^{i}} {w_i^{p^*}(x)\phi_i^{p^{*}}(x)}\mathrm{d}x} \right)^{\frac{p}{p^*}}   \\
            \le&cr_{j+1}^{sp-2n+Q}\left( {\left| {x_0} \right|_h+r_{j+1}} \right)^{\frac{m(n-h)(sp-2n+Q)}{2n-Q}}   \\
            &\times\int_{B^i}\int_{B^i} {\left| {d_{C}(x,y)} \right|^{-Q-sp}\left| {w_{i}(x)\phi_i(x)-w_{i}(y)\phi_i(y)} \right|^{p}}\mathrm{d}x\mathrm{d}y.
        \end{aligned}
    \end{equation}
    Recalling that $\left| {\nabla_{\mathbb{G}^{n}}\phi_i} \right|<c2^ir_{j+1}^{-1}$, we can handle the first term on the right-hand side of \eqref{use Caccioppoli inequality 1 in proof of holder continuity} as
    \begin{equation}
        \begin{aligned}\label{right-hand 1st term estimate 1 in proof of holder continuity}
            &r_{j+1}^{sp}\int_{B^i}\int_{B^i} {\left| {d_{C}(x,y)} \right|^{-Q-sp}{\left( {\mathrm{max}\left\{ {w_{i}(x),w_{i}(y)} \right\}} \right)}^{p}\left| {\phi_i(x)-\phi_i(y)} \right|^{p}}\mathrm{d}x\mathrm{d}y   \\
            \le&2^{ip}r_{j+1}^{sp}r_{j+1}^{-p}k_i^p\int_{B^i\cap\left\{ {u_j\le k_i} \right\}}\int_{B^i} {\left| {d_{C}(x,y)} \right|^{-Q-sp+p}}\mathrm{d}x\mathrm{d}y   \\
            \le&c2^{ip}\left[ {\epsilon\omega\left( {r_j} \right)} \right]^{p}r_{j+1}^{n-Q}\left( {\left| {x_0} \right|_h+r_{j+1}} \right)^{m(n-h)}\left| {B^i\cap\left\{ {u_j\le k_i} \right\}} \right|.
        \end{aligned}
    \end{equation}
    Moreover,
    \begin{equation}\label{right-hand 2nd.1 term estimate 1 in proof of holder continuity}
        \int_{B^i} {w_{i}(x)\phi_i^{p}(x)}\mathrm{d}x\le c\left[ {\epsilon\omega\left( {r_j} \right)} \right]\left| {B^i\cap\left\{ {u_j\le k_i} \right\}} \right|.
    \end{equation}
    Now, we observe that for any $x\in\mathbb{G}^n\backslash B^i$, we have
    \begin{equation*}
        \mathop{\inf}_{y\in\tilde{B}^i}d_c(x,y)\ge d_c\left( {x,x_0} \right)\mathop{\inf}_{y\in\tilde{B}^i}\left( {1-\frac{d_c\left( {x_0,y} \right)}{d_c\left( {x,x_0} \right)}} \right)\ge d_c\left( {x,x_0} \right)\left( {1-\frac{\tilde{\varrho}_i}{\varrho_i}} \right)\ge 2^{-i-1}d_c\left( {x,x_0} \right)
    \end{equation*}
    and obtain the fact that
    \begin{equation*}
        B\left( {x_0,r_{j+1}} \right)\equiv B_{j+1}\subset B^i\Rightarrow\mathbb{G}^n\backslash B^i\subset\mathbb{G}^n\backslash B_{j+1}.
    \end{equation*}
    Consequently, we have
    \begin{equation}\label{right-hand 2nd.2 term step1 estimate 1 in proof of holder continuity}
        r_{j+1}^{sp}\left( {\mathop{\sup}_{y\in\tilde{B}^i}\int_{{\mathbb{G}}^{n}\backslash B^i} {\left| {d_{C}(x,y)} \right|^{-Q-sp}w_{i}^{p-1}(x)}\mathrm{d}x} \right)\le2^{i(Q+sp)}\left[ {\mathrm{Tail} \left( {w_i;x_{0},r_{j+1}} \right)} \right]^{p-1}.
    \end{equation}
    Recalling \eqref{tail < omega in proof of holder continuity} and the facts that $w_i\le2\epsilon\omega\left( {r_j} \right)$ in $B_j$ and $w_i\le\left| {u_j} \right|+2\epsilon\omega\left( {r_j} \right)$ in $\mathbb{G}^n$, we further get
    \begin{equation}\label{right-hand 2nd.2 term step2 estimate 1 in proof of holder continuity}
        \begin{aligned}
            &\left[ {\mathrm{Tail} \left( {w_i;x_{0},r_{j+1}} \right)} \right]^{p-1}   \\
            \le&cr_{j+1}^{sp}\int_{B_j\backslash B_{j+1}} {w_i^{p-1}(x)\left| {d_C\left( {x,x_0} \right)} \right|^{-Q-sp}}\mathrm{d}x+c\left( {\frac{r_{j+1}}{r_j}} \right)^{sp}\left[ {\mathrm{Tail} \left( {w_i;x_{0},r_{j}} \right)} \right]^{p-1}   \\
            \le&cr_{j+1}^{n-Q}\left( {\left| {x_0} \right|_h+r_{j}} \right)^{m(n-h)}\epsilon^{p-1}\omega\left( {r_j} \right)^{p-1}+c\sigma^{sp}\left[ {\mathrm{Tail} \left( {w_i;x_{0},r_{j}} \right)} \right]^{p-1}   \\
            \le&c\left[ {r_{j+1}^{n-Q}\left( {\left| {x_0} \right|_h+r_{j}} \right)^{m(n-h)}+\frac{\sigma^{sp-\alpha(p-1)}}{\epsilon^{p-1}}} \right]\left[ {\epsilon\omega\left( {r_j} \right)} \right]^{p-1}   \\
            =&c\left[ {r_{j+1}^{n-Q}\left( {\left| {x_0} \right|_h+r_{j}} \right)^{m(n-h)}+1} \right]\left[ {\epsilon\omega\left( {r_j} \right)} \right]^{p-1}
        \end{aligned}
    \end{equation}
    by the definition of $\epsilon$ in \eqref{epsilon in proof of holder continuity}, where the constant $c$ depends only on $n,h,p,s,m$ and $\alpha$. Combining the estimate \eqref{right-hand 2nd.2 term step1 estimate 1 in proof of holder continuity} with \eqref{right-hand 2nd.2 term step2 estimate 1 in proof of holder continuity}, we obtain
    \begin{equation}\label{right-hand 2nd.2 term estimate 1 in proof of holder continuity}
        \begin{aligned}
            &r_{j+1}^{sp}\left( {\mathop{\sup}_{y\in\tilde{B}^i}\int_{{\mathbb{G}}^{n}\backslash B^i} {\left| {d_{C}(x,y)} \right|^{-Q-sp}w_{i}^{p-1}(x)}\mathrm{d}x} \right)   \\
            \le& c2^{i(Q+sp)}\left[ {r_{j+1}^{n-Q}\left( {\left| {x_0} \right|_h+r_{j+1}} \right)^{m(n-h)}+1} \right]\left[ {\epsilon\omega\left( {r_j} \right)} \right]^{p-1}.
        \end{aligned}
    \end{equation}
    By combining \eqref{use Caccioppoli inequality 1 in proof of holder continuity}, \eqref{left-hand term estimate 1 in proof of holder continuity}, \eqref{right-hand 1st term estimate 1 in proof of holder continuity}, \eqref{right-hand 2nd.1 term estimate 1 in proof of holder continuity} and \eqref{right-hand 2nd.2 term estimate 1 in proof of holder continuity}, we arrive at
    \begin{equation*}
        A_{i+1}^{\frac{p}{p^*}}\left( {k_i-k_{i+1}} \right)^p\le c\left( {\left| {x_0} \right|_h+r_{j+1}} \right)^{\frac{m(n-h)sp}{2n-Q}}\left[ {2\left( {\left| {x_0} \right|_h+r_{j+1}} \right)^{m(n-h)}+1} \right]2^{i(Q+sp+p)}\left[ {\epsilon\omega\left( {r_j} \right)} \right]^{p}A_i,
    \end{equation*}
    which yields, recalling that $r_{j+1}<r$,
    \begin{equation}\label{iteration1 in proof of holder continuity}
        A_{i+1}\le c\left\{ {\left( {\left| {x_0} \right|_h+r} \right)^{\frac{m(n-h)sp}{2n-Q}}\left[ {2\left( {\left| {x_0} \right|_h+r} \right)^{m(n-h)}+1} \right]} \right\}^{\frac{p^*}{p}}2^{i(Q+sp+2p)\frac{p^*}{p}}A_i^{1+\beta}
    \end{equation}
    with $\beta:=sp/(2n-Q-sp)$ according to the definition of $k_i$'s. Now, we recall that if we establish the following estimate on $A_0$, where
    \begin{equation}\label{A0 require1 in proof of holder continuity}
        \begin{aligned}
            A_0=&\frac{\left| {\tilde{B}\cap\left\{ {u_j\le 2\epsilon\omega\left( {r_j} \right)} \right\}} \right|}{\left| {\tilde{B}} \right|}   \\
            \le&c^{-\frac{1}{\beta}}\left\{ {\left( {\left| {x_0} \right|_h+r} \right)^{\frac{m(n-h)sp}{2n-Q}}\left[ {2\left( {\left| {x_0} \right|_h+r} \right)^{m(n-h)}+1} \right]} \right\}^{-\frac{p^*}{p\beta}}2^{-\frac{(Q+sp+2p)p^*}{p\beta^2}}   \\
            =:&\nu^*,
        \end{aligned}
    \end{equation}
    then by the iteration relation \eqref{iteration1 in proof of holder continuity}, we
    have
    \begin{equation*}
        A_i\le c^{-\frac{1}{\beta}}\left\{ {\left( {\left| {x_0} \right|_h+r} \right)^{\frac{m(n-h)sp}{2n-Q}}\left[ {2\left( {\left| {x_0} \right|_h+r} \right)^{m(n-h)}+1} \right]} \right\}^{-\frac{p^*}{p\beta}}2^{-\frac{(Q+sp+2p)p^*(i\beta+1)}{p\beta^2}},
    \end{equation*}
    which implies $A_j\to0$ as $j\to\infty$. Indeed, the condition \eqref{A0 require1 in proof of holder continuity} is guaranteed by \eqref{c/k < clog in proof of holder continuity} with the choice
    \begin{equation*}
        \sigma=\min\left\{ {\frac{1}{4},e^{-\frac{c_{\log}}{\nu^*}}} \right\},
    \end{equation*}
    which depends only on $n,h,p,s,m,\Omega$ and the difference of $sp/(p-1)$ and $\alpha$. In other words, we have demonstrated that
    \begin{equation*}
        \mathop{\mathrm{osc}}_{B_{j+1}}u\le\left( {1-\epsilon} \right)\omega\left( {r_j} \right)=\left( {1-\epsilon} \right)\left( {\frac{r_j}{r_{j+1}}} \right)^{\alpha}\omega\left( {r_{j+1}} \right)=\left( {1-\epsilon} \right)\sigma^{-\alpha}\omega\left( {r_{j+1}} \right).
    \end{equation*}
    Finally, by taking $\alpha\in\left( {0,\frac{sp}{p-1}} \right)$ small enough such that
    \begin{equation*}
        \sigma^{\alpha}\ge1-\epsilon=1-\sigma^{\frac{sp}{p-1}-\alpha},
    \end{equation*}
    then, it is clear that $\alpha$ depends only on $n,h,p,s,m,\Omega$ and
    \begin{equation*}
        \mathop{\mathrm{osc}}_{B_{j+1}}u\le\omega\left( {r_{j+1}} \right).
    \end{equation*}
    This completes the induction step and concludes the proof for the case $sp<2n-Q$.

    For the remaining case, namely when $sp=2n-Q$, we proceed similarly to the proof of local boundedness.
    Choose $0<s_1<s<1$, and we are justified in applying the H{\"o}lder inequality, which yields, for any $p<q<p^{*}_1=\frac{(2n-Q)p}{2n-Q-s_1p}$,
    \begin{equation}\label{left-hand term estimate 2 in proof of holder continuity}
        \begin{aligned}
            &A_{i+1}^{\frac{p}{q}}\left( {k_i-k_{i+1}} \right)^p   \\
            =&\frac{1}{\left| {B^{i+1}} \right|^{\frac{p}{q}}}\left( {\int_{B^{i+1}\cap\left\{ {u_j\le k_{i+1}} \right\}} {\left( {k_i-k_{i+1}} \right)^{q}\phi_i^{q}(x)}\mathrm{d}x} \right)^{\frac{p}{q}}   \\
            \le&\frac{1}{\left| {B^{i+1}} \right|^{\frac{p}{q}}}\left( {\int_{B^{i}} {w_i^{q}(x)\phi_i^{q}(x)}\mathrm{d}x} \right)^{\frac{p}{q}}   \\
            \le&\frac{\left| {B^{i}} \right|^{\frac{\left( {p_1^*-q} \right)p}{qp_1^*}}}{\left| {B^{i+1}} \right|^{\frac{p}{q}}}\left( {\int_{B^{i}} {w_i^{p_1^*}(x)\phi_i^{p_1^*}(x)}\mathrm{d}x} \right)^{\frac{p}{p_1^*}}   \\
            \le&cr_{j+1}^{s_1p-2n+Q}\left( {\left| {x_0} \right|_h+r_{j+1}} \right)^{\frac{m(n-h)(s_1p-2n+Q)}{2n-Q}}   \\
            &\times\int_{B^i}\int_{B^i} {\left| {d_{C}(x,y)} \right|^{-n-s_1p}\left| {w_{i}(x)\phi_i(x)-w_{i}(y)\phi_i(y)} \right|^{p}}\mathrm{d}x\mathrm{d}y.
        \end{aligned}
    \end{equation}
    The term on the right-hand side of \eqref{left-hand term estimate 2 in proof of holder continuity} can be estimated using the nonlocal Caccioppoli inequality, as in \eqref{use Caccioppoli inequality 1 in proof of holder continuity}, along with the analogous estimates in \eqref{right-hand 1st term estimate 1 in proof of holder continuity}, \eqref{right-hand 2nd.1 term estimate 1 in proof of holder continuity} and \eqref{right-hand 2nd.2 term estimate 1 in proof of holder continuity}. Thus, we obtain
    \begin{equation*}
        \begin{aligned}
            &A_{i+1}^{\frac{p}{q}}\left( {k_i-k_{i+1}} \right)^p   \\
            \le&c\left( {\left| {x_0} \right|_h+r_{j+1}} \right)^{\frac{m(n-h)s_1p}{2n-Q}}\left[ {2\left( {\left| {x_0} \right|_h+r_{j+1}} \right)^{m(n-h)}+1} \right]2^{i(Q+s_1p+p)}\left[ {\epsilon\omega\left( {r_j} \right)} \right]^{p}A_i,
        \end{aligned}
    \end{equation*}
    which yields, recalling that $r_{j+1}<r$,
    \begin{equation}\label{iteration2 in proof of holder continuity}
        A_{i+1}\le c\left\{ {\left( {\left| {x_0} \right|_h+r_{j+1}} \right)^{\frac{m(n-h)s_1p}{2n-Q}}\left[ {2\left( {\left| {x_0} \right|_h+r_{j+1}} \right)^{m(n-h)}+1} \right]} \right\}^{\frac{q}{p}}2^{i(Q+s_1p+2p)\frac{q}{p}}A_i^{1+\beta}
    \end{equation}
    with $\beta:=q/p-1>0$ according to the definition of $k_i$'s. Now, we recall that if we establish the following estimate on $A_0$, where
    \begin{equation}\label{A0 require2 in proof of holder continuity}
        \begin{aligned}
            A_0=&\frac{\left| {\tilde{B}\cap\left\{ {u_j\le 2\epsilon\omega\left( {r_j} \right)} \right\}} \right|}{\left| {\tilde{B}} \right|}   \\
            \le&c^{-\frac{1}{\beta}}\left\{ {\left( {\left| {x_0} \right|_h+r_{j+1}} \right)^{\frac{m(n-h)s_1p}{2n-Q}}\left[ {2\left( {\left| {x_0} \right|_h+r_{j+1}} \right)^{m(n-h)}+1} \right]} \right\}^{-\frac{q}{p\beta}}2^{-\frac{(Q+s_1p+2p)q}{p\beta^2}}   \\
            =:&\bar{\nu}.
        \end{aligned}
    \end{equation}
    Then by the iteration relation \eqref{iteration2 in proof of holder continuity}, we can deduce that
    \begin{equation*}
        A_i\le c^{-\frac{1}{\beta}}\left\{ {\left( {\left| {x_0} \right|_h+r_{j+1}} \right)^{\frac{m(n-h)s_1p}{2n-Q}}\left[ {2\left( {\left| {x_0} \right|_h+r_{j+1}} \right)^{m(n-h)}+1} \right]} \right\}^{-\frac{q}{p\beta}}2^{-\frac{(Q+s_1p+2p)q(i\beta+1)}{p\beta^2}},
    \end{equation*}
    which implies $A_j\to0$ as $j\to\infty$. Similarly, the condition \eqref{A0 require2 in proof of holder continuity} is guaranteed by \eqref{c/k < clog in proof of holder continuity} with the choice
    \begin{equation*}
        \sigma=\min\left\{ {\frac{1}{4},e^{-\frac{c_{\log}}{\bar{\nu}}}} \right\}.
    \end{equation*}
    As in the case $sp<2n-Q$, this completes the induction step and concludes the proof for the case $sp=2n-Q$.
\end{proof}
\end{section}

\subsection*{Acknowledgements}    {Yu Liu was supported by the Beijing Natural
    Science Foundation of China (No.\,1232023) and the National
Natural Science Foundation of China (No.\,12471089, No.\,12271042).}
Shaoguang Shi was supported  supported by the National Natural
Science Foundation of China (No.\,12271232, No.\,12071197).

\subsection*{Data availability}  No data was used for the research described in the article.

\bibliographystyle{amsplain}

\end{document}